\documentclass{ws-jta}

 \usepackage{xcolor}
 \usepackage{amsmath}
 \usepackage{tcolorbox}
 \usepackage{multicol}
 \usepackage{amscd}
  \usepackage{graphicx,subfigure,epsfig}
 \usepackage{mathtools}

 \newtheorem{question}{Question}[section]
 \usepackage{subcaption}
 \usepackage{epsfig}
 \usepackage{xcolor}
 \usepackage{soul}
 \usepackage{tikz}
 \usetikzlibrary{arrows.meta}
 \usetikzlibrary{shapes.geometric}
 \usepackage{tcolorbox,amsmath,parcolumns}
\begin{document}

\markboth{Kirandeep Kaur, Madeti Prabhakar, Vaibhav Keshari}{$\mathcal{B}$-Index polynomial invariants of twisted knots}

 \catchline{}{}{}{}{}

\title{$\mathcal{B}$-INDEX POLYNOMIAL FOR TWISTED KNOTS }
\author{Kirandeep Kaur\textsuperscript{1}, Madeti Prabhakar\textsuperscript{2} and Vaibhav Keshari\textsuperscript{3} }

\address{\textsuperscript{1}Department of Computational, Statistics and Data Analytics, Guru Nanak Dev University, Punjab 1430005, India,
\email{\textsuperscript{1} kirandeepoffical@gmail.com}
\textsuperscript{2,3} Department of Mathematics, Indian Institute of Technology Ropar, India,
\email{\textsuperscript{2}prabhakar@iitrpr.ac.in, \textsuperscript{3} vaibhav.23maz0022@iitrpr.ac.in}}

\maketitle 
\begin{abstract}
 
A twisted link is a generalization of a virtual link associated with link diagrams on closed surfaces, which may be non-orientable. In this paper, we generalize the notion of the index value for twisted knots. Based on this generalization, we introduce a polynomial invariant for twisted knots, called the $\mathcal{B}$-Index polynomial. Furthermore, we construct a family of twisted knots $\{D_n\}_{n\geq 1}$ with arc shift number $n$ and determine their $\mathcal{B}$-Index polynomials explicitly in terms of $n$. These results demonstrate the effectiveness and sensitivity of the $\mathcal{B}$-Index polynomial as an invariant of twisted knots. We also study the behavior of this polynomial under mirror images and orientation reversal. Furthermore, we conclude this paper by investigating the cosmetic crossing change conjecture and establishing a condition under which a crossing does not admit cosmetic behavior.
\end{abstract}




\section{Introduction}

 Virtual knots were introduced by L.~Kauffman in \cite{kauffman1999virtual} as an extension of classical knots, representing them through diagrams that include both classical and virtual crossings. Virtual links may be interpreted as stable equivalence classes of links embedded in oriented $3$-manifolds that are line bundles over closed oriented surfaces~\cite{Carter,N K Kamada}.\\
 
 \noindent In \cite{Bourgoin}, M. O. Bourgoin introduced twisted links, defined as stable equivalence classes of links in oriented 3-manifolds that are line bundles over closed, possibly non-orientable surfaces. They are represented by twisted link diagrams that include both classical and virtual crossings, as well as bars on the arcs. An example of a twisted knot diagram is illustrated in Fig.~\ref{fig:twistedknot}.\\





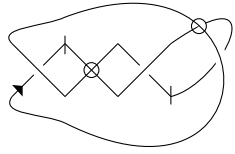
\begin{figure}[htbp]
    \centering
     \begin{tikzpicture}[scale=0.7]
        \draw (0,0) -- (1,1);
        \draw[white, line width=5pt] (0,1) -- (1,0);
        \draw (0,1) -- (1,0);
        \draw (1,1)--(2,0);
        \draw (1,0) -- (2,1);
        \draw (1.5,0.5) circle (4pt);
        \draw (2,1) -- (3,0);
        \draw[white, line width=5pt] (2.1,0.1) -- (3,1);
        \draw (2,0) -- (3,1);
        \draw (3,1) .. controls (5,2.4) and (4,0.25) .. (3,0);
        \draw[white, line width=5pt] (3.53,1.33) to [out=-45, in=80] (4,0);
        \draw(1,1.5) .. controls (5,3) and (5,-2.5) .. (1,-0.5);
        \draw (0,1) to[out=135 , in=205 ] (1,1.5);
        \draw (0,0) to[out=235 , in=155 ] (1,-0.5);
        \draw (3.53,1.33) circle (4pt); 
        \draw (1,0.8) -- (1,1.15);
        \draw (3,0.2) -- (3,-0.15);
        \fill (0,0.22) -- (0.2,0.2) -- (0.2,0)-- cycle; 
         
        \end{tikzpicture}
    \caption{Twisted knot diagram.}
    \label{fig:twistedknot}
\end{figure}

\noindent Although twisted links naturally generalize virtual links, not all invariants of virtual links extend to twisted links. Furthermore, compared to virtual links, there are relatively few known invariants for twisted links.
We are mainly interested in invariants of polynomial type, as they constitute an important class of invariants for twisted links. 
Several polynomial invariants originally defined for virtual links have been generalized to twisted links, for instance, the Jones polynomial \cite{Bourgoin},  multivariable polynomial \cite{KamadaMulti},  the Miyazawa
polynomial \cite{KamadaMiyazawa}, and  surface bracket polynomial \cite{Kamadasurface}, etc.\\

\noindent In this paper, our aim is to introduce a new polynomial invariant for twisted knots. Our construction is motivated by  the index value of a crossing, defined in \cite{cheng2013polynomial}, and by Kauffman
affine index polynomial $P_K(t)$ \cite{kauffman2013affine}. Recall that for an oriented virtual knot $K$ represented by a diagram $D$, the affine index polynomial is defined by
\[P_D(t)=\displaystyle \sum_{c} \operatorname{sgn}(c)(t^{W_D(c)} -1),\] 
where $ \operatorname{sgn}(c) $ denotes the sign of the crossing $c$, $W_D(c)$ is the weight associated with the crossing $c$, and the sum is taken over all classical crossings of the diagram.
For each crossing $c$, the weight $W_D(c)$ defined in \cite{kauffman2013affine} coincides with the index value $\operatorname{Ind}(c)$ introduced in \cite{cheng2013polynomial}. These weights can also be computed using the wriggle number, an invariant of two-component virtual links (see \cite{folwaczny2013linking}).\\

\noindent We introduce two new indices, $\operatorname{Ind}^1(c)$ and $\operatorname{Ind}^2(c)$, associated with a crossing $c$, which are motivated by and generalize the method used to compute the wriggle number in the virtual setting ~\cite{folwaczny2013linking}. These indices enable us to extend the affine index polynomial to twisted knots and to define a new polynomial invariant, denoted by $\mathcal{B}_D(t)$, which we call the $\mathcal{B}$-Index polynomial. The nomenclature is specifically chosen to reflect the invariant's fundamental dependence on the bars present in twisted knot diagrams. This dependence allows the polynomial to capture topological information inherent to the non-orientable nature of the underlying surface, which is often lost in broader generalizations. A key observation of our construction is the polynomial's ability to distinguish between the twisted and virtual categories. Specifically, the $\mathcal{B}$-Index polynomial is trivial (zero) for all virtual knots. This vanishing property positions $\mathcal{B}_D(t)$ as a dedicated tool for studying the ``twistedness" of a knot, effectively filtering out classical and virtual components to isolate features unique to the twisted setting. \\

\noindent The paper is organized as follows. In Section \ref{pre}, we briefly recall the basic definitions  and concepts of twisted links that are required to establish the main results of the paper. In Section \ref{sec:pol}, we define the index values of a crossing and study the behavior of these index values under classical Reidemeister moves. We also introduce the $\mathcal{B}$-Index polynomial and establish its invariance. Furthermore, in \cite{KamadaKawataOkuboShimizu}, Kamada et al.\ posed the following question of whether the twisted knot diagram $3_{82}$ is equivalent to the trivial non-orientable curve (a
trivial loop with a bar). Since several known invariants, including the $X$-polynomial invariant, the twisted JKSS invariant, the twisted $3$-coloring number, twisted biquandle colorings, and the $X$-polynomial of the double covering diagram, fail to distinguish these diagrams, the question remained open. Using the $\mathcal{B}$-Index polynomial, we show that $3_{82}$ is not equivalent to the trivial non-orientable curve, thereby providing an answer to this question.\\

\noindent In Section~\ref{sec:pol}, we also construct a family of twisted knots $D_n$ with arc shift number $n$ and determine their $\mathcal{B}$-Index polynomials explicitly. The resulting formula depends on $n$, showing that the $\mathcal{B}$-Index polynomial detects the arc shift number for this family of twisted knots. Further, we examine the behavior of the $\mathcal{B}$-Index polynomial under mirror image and orientation reversal. 
In Section~\ref{sec:cs}, we investigate the cosmetic crossing change conjecture for twisted knots and provide a specific condition under which a crossing $c$ is not a cosmetic crossing.

\section{Preliminaries}\label{pre}
Twisted knots were  introduced by M. O. Bourgoin \cite{Bourgoin} as a generalization of virtual knots. They are represented by twisted knot diagrams that contain classical crossings, virtual crossings, and bars placed on the arcs of the diagram. Two twisted knot diagrams are considered equivalent if they are related by a finite sequence of classical Reidemeister moves, virtual Reidemeister moves, and twisted Reidemeister moves, shown in Fig.~\ref{fig:cls}, Fig.~\ref{fig:rei}, and Fig.~\ref{fig:Re:tw}, respectively.\\

\begin{figure}[htbp]
    \centering
        \begin{tikzpicture}[scale=1.4]
            \draw (-0.5,0.5) -- (0,-0.3);
            \draw[white, line width=5pt] (-0.5,-0.5) -- (0,0.3);
            \draw (-0.5,-0.5) -- (0,0.3);
            \draw (0,-0.3) to[out=45, in=-45] (0,0.3);
            \draw (0.3,0) -- (1,0);
            \draw(0.3,0) -- (0.4,0.1);
            \draw (0.3,0) -- (0.4,-0.1);
            \draw(0.9,0.1) --(1,0);
            \draw(0.9,-0.1) --(1,0);
            \node at (0.65,0.15) {\tiny $R_1$};
            \draw (1,0.5) .. controls (1.3,0) .. (1,-0.5);
            \draw (2,0.5) .. controls (2.25,0) .. (2,-0.5);
            \draw (2.5,0.5) .. controls (2.25,0) .. (2.5,-0.5); 
            \draw (2.55,0) -- (3.25,0);
            \draw(2.55,0) -- (2.65,0.1);
            \draw (2.55,0) -- (2.65,-0.1);
            \draw(3.15,0.1) --(3.25,0);
            \draw(3.15,-0.1) --(3.25,0);
            \node at (2.9,0.15) {\tiny $R_2$};
             \draw (3.4,0.5) .. controls (3.7,0) .. (3.4,-0.5);
             \draw[white, line width=3pt] (3.7,0.5) .. controls (3.4,0) .. (3.7,-0.5); 
            \draw (3.7,0.5) .. controls (3.4,0) .. (3.7,-0.5); 
             \draw (4.5,0.5) -- (5.5,-0.5);
            \draw[white, line width=4pt] (4.5,-0.5) -- (5.5,0.5);
            \draw (4.5,-0.5) -- (5.5,0.5);
             \draw[white, line width=4pt] (4.4,0.25) -- (5.6,0.25);
            \draw (4.4,0.25) -- (5.6,0.25);
            \draw (5.6,0) -- (6.3,0);
            \draw(5.6,0) -- (5.7,0.1);
            \draw (5.6,0) -- (5.7,-0.1);
            \draw(6.3,0) --(6.2,0.1);
            \draw(6.3,0) --(6.2,-0.1);
            \node at (5.95,0.15) {\tiny $R_3$};
             \draw (6.5,0.5) -- (7.5,-0.5);
            \draw[white, line width=4pt] (6.5,-0.5) -- (7.5,0.5);
            \draw (6.5,-0.5) -- (7.5,0.5);
             \draw[white, line width=4pt] (6.4,-0.25) -- (7.6,-0.25);
            \draw (6.4,-0.25) -- (7.6,-0.25);

        \end{tikzpicture}
        \caption{\small  Classical Reidemeister moves.}
        \label{fig:cls}
    \end{figure}
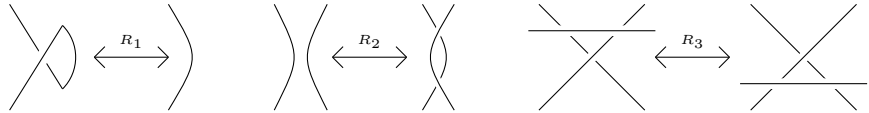
    \begin{figure}
    \centering
        \begin{tikzpicture}[scale=1.4]
        \begin{scope}[xshift=1cm]
            \draw (-0.5,0.5) -- (0,-0.3);
            \draw (-0.2,0) circle (2.5pt);
            \draw (-0.5,-0.5) -- (0,0.3);
            \draw (0,-0.3) to[out=45, in=-45] (0,0.3);
           \draw (0.3,0) -- (1,0);
            \draw(0.3,0) -- (0.4,0.1);
            \draw (0.3,0) -- (0.4,-0.1);
            \draw(0.9,0.1) --(1,0);
            \draw(0.9,-0.1) --(1,0);
            \node at (0.6,0.15) {\tiny $VR_1$};
            \draw (1,0.5) .. controls (1.3,0) .. (1,-0.5);
            \draw (2,0.5) .. controls (2.25,0) .. (2,-0.5);
            \draw (2.5,0.5) .. controls (2.25,0) .. (2.5,-0.5); 
             \draw (2.55,0) -- (3.25,0);
            \draw(2.55,0) -- (2.65,0.1);
            \draw (2.55,0) -- (2.65,-0.1);
            \draw(3.15,0.1) --(3.25,0);
            \draw(3.15,-0.1) --(3.25,0);
            \node at (2.9,0.15) {\tiny $VR_2$};
             \draw (3.4,0.5) .. controls (3.7,0) .. (3.4,-0.5);
             \draw (3.55,0.25) circle (2.5pt); 
             \draw (3.55,-0.25) circle (2.5pt); 
            \draw (3.7,0.5) .. controls (3.4,0) .. (3.7,-0.5); 
            \end{scope}
            \draw (3.2,-1) -- (4.2,-2);
            \draw (3.45,-1.25) circle (2.5pt);
            \draw[white, line width=4pt] (3.2,-2) -- (4.2,-1);
            \draw (3.2,-2) -- (4.2,-1);
           \draw (3.95,-1.25) circle (2.5pt);
            \draw (3.1,-1.25) -- (4.3,-1.25);
            \draw (4.3,-1.5) -- (4.9,-1.5);
            \draw(4.8,-1.4) -- (4.9,-1.5);
            \draw (4.8,-1.6) -- (4.9,-1.5);
            \draw(4.3,-1.5) --(4.4,-1.4);
            \draw(4.3,-1.5) --(4.4,-1.6);
            \node at (4.6,-1.35) {\tiny $SV$};
             \draw (5,-1) -- (6,-2);
            \draw[white, line width=4pt](5,-2) -- (6,-1);
            \draw (5,-2) -- (6,-1);
             \draw (5.25,-1.75) circle (2.5pt);
             \draw (5.75,-1.75) circle (2.5pt); 
            \draw (4.9,-1.75) -- (6.1,-1.75);
            \draw (-0.5,-1) -- (0.5,-2);
             \draw (0,-1.5) circle (2.5pt);
            \draw (-0.5,-2) -- (0.5,-1);
             \draw (-0.28,-1.25) circle (2.5pt);
             \draw (0.28,-1.25) circle (2.5pt);
            \draw (-0.6,-1.25) -- (0.6,-1.25);
            \draw (0.7,-1.5) -- (1.3,-1.5);
            \draw(0.7,-1.5) -- (0.8,-1.4);
            \draw (0.7,-1.5) -- (0.8,-1.6);
            \draw(1.3,-1.5) --(1.2,-1.4);
            \draw(1.3,-1.5) --(1.2,-1.6);
            \node at (0.98,-1.35) {\tiny $VR3$};
             \draw (1.5,-1) -- (2.5,-2);
            \draw(2,-1.5) circle (2.5pt);
            \draw (1.5,-2) -- (2.5,-1);
             \draw (1.75,-1.75) circle (2.5pt);
             \draw (2.28,-1.75) circle (2.5pt);
            \draw (1.4,-1.75) -- (2.6,-1.75);
        \end{tikzpicture}
    \caption{\small  Virtual Reidemeister moves.}
    \label{fig:rei}
\end{figure}
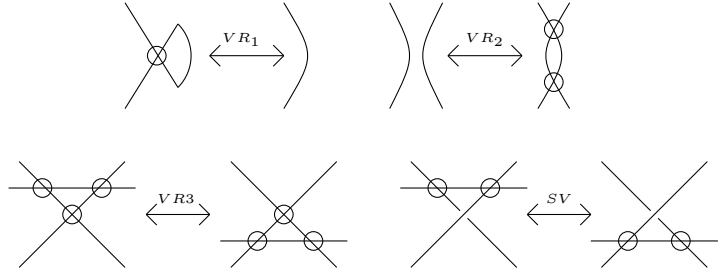

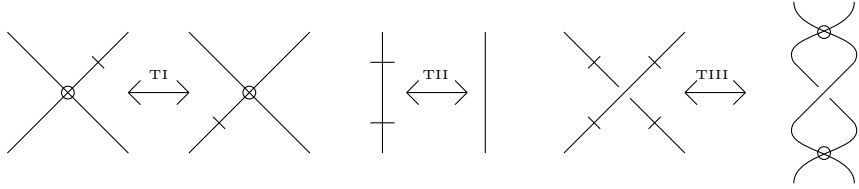
\begin{figure}
    \centering
   \begin{tikzpicture}[scale=0.8]
    \draw (-1,1) -- (1,-1);
    \draw (-1,-1) -- (1,1);
    \draw (0,0) circle (3pt);
    \draw (0.4,0.6) -- (0.6,0.4);
\draw (1,0) --(2,0);
\draw (1.8,0.2) --(2,0);
\draw (1.8,-0.2) --(2,0);
\draw (1,0) --(1.2,0.2);
\draw (1,0) --(1.2,-0.2);
 \node at (1.5,0.3) {\tiny TI};
 
\draw (2,1) -- (4,-1);
    \draw (2,-1) -- (4,1);
    \draw (3,0) circle (3pt);
    \draw (2.4,-0.4) -- (2.6,-0.6);

    \draw (5.2,1) -- (5.2,-1);
    \draw (5,0.5) -- (5.4,0.5);
     \draw (5,-0.5) -- (5.4,-0.5);
     \draw (5.6,0) -- (6.6,0);
    \draw (5.8,0.2) --(5.6,0);
    \draw (5.8,-0.2) --(5.6,0);
\draw (6.6,0) --(6.4,0.2);
\draw (6.6,0) --(6.4,-0.2);
\draw (6.9,1) -- (6.9,-1);
\node at (6.08,0.3) {\tiny TII};

\draw (8.2,1) -- (10.2,-1);
\draw[white, line width =6pt] (8.2,-1) -- (10.2,1);
\draw (8.2,-1) -- (10.2,1);
\draw(8.6,0.4) -- (8.8,0.6);
\draw(8.6,-0.4) -- (8.8,-0.6);
\draw(9.8,0.4) -- (9.6,0.6);
\draw(9.8,-0.4) -- (9.6,-0.6);  
\draw (10.2,0) -- (11.2,0);
\draw (10.2,0) -- (10.4,0.2);
\draw (10.2,0) -- (10.4,-0.2);
\draw (11,0.2) -- (11.2,0);
\draw (11,-0.2) -- (11.2,0);
\node at (10.65,0.3) {\tiny TIII};

\draw (13,1.5) to[out=-90  , in=120 ] (12,0.5);
\draw (13,0.5) to[out=60  , in=-90 ] (12,1.5);
\draw (12,0.5) -- (13,-0.5);
\draw[white, line width=6pt] (12,-0.5) -- (13,0.5);
\draw (12,-0.5) -- (13,0.5);
\draw (13,-1.5) to[out=90  , in=-120 ] (12,-0.5);
\draw (13,-0.5) to[out=-60  , in=90 ] (12,-1.5);
 \draw (12.5,1) circle (3pt);
 \draw (12.5,-1) circle (3pt);
\end{tikzpicture}
    \caption{\small Twisted Reidemeister moves.}
    \label{fig:Re:tw}
\end{figure}

\noindent A trivial twisted knot is either a trivial knot or a trivial knot with a bar. Let $D$ be an oriented twisted knot diagram, and let $c$ be a crossing of $D$. The sign of the crossing $c$, determined by the orientation of the arcs, is assigned a weight of either $+1$ or $-1$, as shown in Fig.~\ref{fig2}. \\  
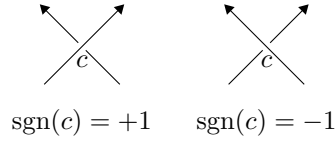
\begin{figure}
    \centering
    \begin{subfigure}{}
        \begin{tikzpicture}
            \draw (0,1) -- (1,0);
            \draw[white, line width=4pt] (0,0)-- (1,1);
            \draw (0,0)-- (1,1);
            \node at (0.5,0.3) {$c$};
            \node at (0,1) [
  fill=black,
  regular polygon,
  regular polygon sides=3,
  rotate=45,
   scale=0.25
] {};
\node at (1,1) [
  fill=black,
  regular polygon,
  regular polygon sides=3,
  rotate=-45,
   scale=0.25
] {};
\node at (0.5,-0.5) {$\operatorname{sgn}(c)=+1$};
        \end{tikzpicture}
    \end{subfigure}
    \begin{subfigure}{}
        \begin{tikzpicture}
            \draw (0,0)-- (1,1);
            \draw[white, line width=4pt] (0,1) -- (1,0);
            \draw (0,1) -- (1,0);
            \node at (0.5,0.3) {$c$};
            \node at (0,1) [
  fill=black,
  regular polygon,
  regular polygon sides=3,
  rotate=45,
   scale=0.25
] {};
\node at (1,1) [
  fill=black,
  regular polygon,
  regular polygon sides=3,
  rotate=-45,
   scale=0.25
] {};
\node at (0.5,-0.5) {$\operatorname{sgn}(c)=-1$};
        \end{tikzpicture}
    \end{subfigure}
    \caption{Crossing signs.}
    \label{fig2}
\end{figure}

\noindent Let $D$ be an oriented twisted knot diagram. The diagram $\overline{D}$ is obtained from $D$ by reversing its orientation. The mirror image of $D$, denoted by $D^{*}$, is obtained by interchanging the overcrossing and undercrossing strands at every classical crossing of $D$.\\

\noindent In \cite{komal2024arc}, Komal and Prabhakar introduced a new invariant called the arc shift number for twisted knots, which extends the invariant arc shift number for virtual knots defined in \cite{Gillarc}. 
\noindent An arc in a twisted knot diagram, say $(a, b)$, is defined as a segment that passes through exactly one pair of crossings (classical or virtual), denoted by  $(c_1, c_2)$, where $a$ is incident to $c_1$ and $b$ is incident to $c_2$. Between these crossing points, bars may
or may not appear. 
\noindent An arc shift move on an arc $(a,b)$ is obtained by cutting the twisted knot diagram $D$ at the points $a$ and $b$ and reconnecting the resulting endpoints after interchanging them. Any new intersections created during this local transformation are regarded as virtual crossings, and the orientation of the arc is reversed. The arc shift move on the arc $(a,b)$ is illustrated in Fig.~\ref{arcshift}.\\

\noindent In \cite{komal2024arc}, it was proved that the arc shift is an unknotting operation for twisted knots.
Based on this operation, a twisted knot invariant, called the arc shift number, was defined.\\
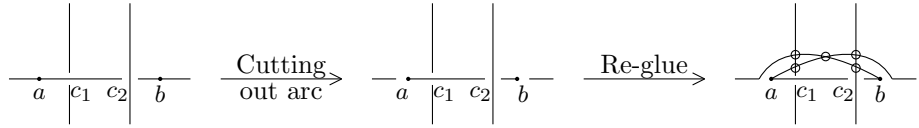
\begin{figure}[htbp]
    \centering
    \begin{tikzpicture}[scale=0.8]
        \draw (-8,-0.25) -- (-5,-0.25);
        \draw (-7,1) -- (-7,-0.15);
        \draw (-7,-0.35) -- (-7,-1);
        \draw[white, line width=6pt] (-6,1) -- (-6,-1);
        \draw (-6,1) -- (-6,-1);
        \fill (-7.5,-0.25) circle (1pt);
        \fill (-5.5,-0.25) circle (1pt);
        \node at (-7.5,-0.5) {\small $a$};
         \node at (-5.5,-0.5) {\small $b$};
          \node at (-6.8,-0.5) {\small $c_1$};
         \node at (-6.2,-0.5) {\small $c_2$};
         \draw (-4.5,-0.25) -- (-2.5,-0.25);
         \draw (-2.5,-0.25) -- (-2.7,-0.15);
         \draw (-2.5,-0.25) -- (-2.7,-0.35);
         \node at (-3.5,-0.05) {Cutting};
         \node at (-3.5,-0.45) {out arc};
         \draw (-2,-0.25) -- (-1.6,-0.25);
         \draw (-1.4,-0.25) -- (0.4,-0.25);
         \draw (0.6,-0.25) -- (1,-0.25);
         \draw (-1,1) -- (-1,-0.15);
          \draw (-1,-0.35) -- (-1,-1);
          \draw[white, line width=6pt] (0,-1) -- (0,1);
         \draw (0,-1) -- (0,1);
         \fill (-1.4,-0.25) circle (1pt);
         \fill (0.4,-0.25) circle (1pt);
         \node at (-1.5,-0.5) {\small $a$};
         \node at (0.5,-0.5) {\small $b$};
          \node at (-0.8,-0.5) {\small $c_1$};
         \node at (-0.2,-0.5) {\small $c_2$};
         \draw (1.5,-0.25) -- (3.5,-0.25);
         \draw (3.3,-0.15) -- (3.5,-0.25);
         \draw (3.3,-0.35) -- (3.5,-0.25);
         \node at (2.5,-0.05) { Re-glue };
         \draw (4,-0.25) -- (4.4,-0.25);
         \draw (4.6,-0.25) -- (6.4,-0.25);
         \draw (6.6,-0.25) -- (7,-0.25);
         \draw (5,1) -- (5,-0.2);
         \draw (5,-0.35) -- (5,-1);
         \draw[white, line width=6pt] (6,1) -- (6,-1);
         \draw (6,1) -- (6,-1);
         \draw (4.4,-0.25) to[in=150 , out=60 ] (6.4,-0.25);
         \draw (4.6,-0.25) to[in=120 , out=30 ] (6.6,-0.25);
         \draw (5,-0.07) circle (2pt);
         \draw (5,0.15) circle (2pt);
         \draw (5.5,0.12) circle (2pt);
         \draw (6,-0.07) circle (2pt);
         \draw (6,0.15) circle (2pt);
         \fill (4.6,-0.25) circle (1pt);
         \fill (6.4,-0.25) circle (1pt);
          \node at (4.6,-0.5) {\small $a$};
         \node at (6.4,-0.5) {\small $b$};
          \node at (5.2,-0.5) {\small $c_1$};
         \node at (5.8,-0.5) {\small $c_2$};
    \end{tikzpicture}
    \caption{An arc shift move. }
\label{arcshift} 
\end{figure}
 \begin{definition}\label{def:arc}\cite{komal2024arc} For any given twisted knot $K$, the arc
shift number, denoted by $A(K)$, is given as
\[A(K)=\min \{A(D)|~D \text{~represents~} K\},\]
 where $A(D)$ is the minimum number of arc shift
moves required to transform a diagram $D$ into a trivial twisted knot.\end{definition}
\begin{theorem}\cite{komal2024arc}
Arc shift number is an invariant for twisted knots.
\end{theorem} 
\noindent The notion of odd crossings and the odd writhe number was introduced by Kauffman in \cite{kauffman2004} for virtual knots, and later extended to twisted knots in \cite{komal2024arc}.  A crossing $c$ in a twisted knot diagram $ D$ is called  odd if, while traversing the diagram starting from 
$c$ and returning to $c$ along a complete path, one encounters an odd number of classical crossings. The odd writhe number $J(D)$ of $D$ is defined as the sum of the
signs of all the odd crossings in $D$.
\begin{theorem}\cite{komal2024arc} Odd writhe $J(K)$ is an invariant for twisted knot $K$.\end{theorem}
\begin{theorem}\cite{komal2024arc}\label{thm:prearc} The odd writhe $J(K)$ forms a lower bound for the arc shift number $A(K)$
of twisted knot $K$. In particular,
\[\left|\dfrac{J(K)}{2}\right|\leq A(K).\]
\end{theorem}

\section{$\mathcal{B}$-Index polynomial of twisted knot}\label{sec:pol}
Let $D$ be an oriented twisted link diagram. In this paper, a twisted link means a two-component twisted link.
  \begin{definition}
Let $D$ be an oriented twisted knot diagram. If $D$ contains at most one bar, then the entire diagram is regarded as a single arc. Otherwise, an \emph{arc} of $D$ is a segment of the diagram between two consecutive bars, taken along the orientation.
\end{definition}

\begin{remark}
    Note that the definition of an arc used in Section~3 and thereafter differs from that in \cite{komal2024arc}. Throughout this paper, an arc in the sense of \cite{komal2024arc} and as used in the definition~\ref{def:arc} of the arc shift number in Section~2 will be referred to as an \emph{arc segment}.
\end{remark}

\begin{definition} Two arcs $C_1$ and $C_2$ of a twisted link diagram $D$ are said to be consecutive if, while traversing along the orientation from $C_1$ to $C_2$, the two arcs share a common bar. For example, arcs $C_1$ and $C_2$ illustrated in Fig.~\ref{fig:exp1} are  consecutive arcs. 
\end{definition}

\noindent A sequence of arcs $\langle C_1,C_2,\ldots,C_n \rangle$ in a twisted link diagram $D$ is said to be consecutive if, for each $i=1,2,\ldots,n-1$, the arcs $C_i$ and $C_{i+1}$ are consecutive.\\

\noindent It follows that any sequence $\langle C_1,C_2,\ldots, C_n \rangle$ of consecutive arcs in a link diagram $D$, beginning at $C_1$, traverses all arcs of the link component containing $C_1$.

\begin{definition} A linking crossing $c$ is said to be an \emph{over} (respectively, \emph{under}) linking crossing of an oriented arc $C$ if, while traversing $C$ along its orientation, the crossing $c$ is encountered as an overcrossing (respectively, undercrossing). The set of all over linking crossings of $C$ is denoted by $\mathcal{O}(C)$, and the set of all under linking crossings of $C$ is denoted by $\mathcal{U}(C)$.
\end{definition}

\begin{definition}
    Let $D$ be a twisted link diagram with an arc $C_1$, and let $\langle C_1,C_2,\ldots, C_n \rangle$ be the sequence of all consecutive arcs of $D$ starting from $C_1$. Then, the set of all over linking crossings of $D$ with respect to arc $C_1$, denoted by $\mathcal{O}_{L}(D:C_1)$, is defined as, 
\[
\mathcal{O}_{L}(D:C_1)=
\begin{cases}
\mathcal{O}(C_1)\cup \mathcal{U}(C_2) \cup \mathcal{O}(C_3)\ldots \cup \mathcal{U}(C_n),~ ~\text{if n is even,}\\
\mathcal{O}(C_1)\cup \mathcal{U}(C_2) \cup \mathcal{O}(C_3)\ldots \cup \mathcal{O}(C_n),~~\text{if n is odd.}
\end{cases}
\]
    
\noindent Similarly, the set of all under linking crossings of $D$ with respect to arc $C_1$, denoted by $\mathcal{U}_{L}(D:C_1)$, is defined as, \\
\[
\mathcal{U}_{L}(D:C_1)=
     \begin{cases}
       \mathcal{U}(C_1)\cup \mathcal{O}(C_2) \cup \mathcal{U}(C_3)\ldots \cup \mathcal{O}(C_n) , &~\text{if ~} n \text{~is even},\\
     \mathcal{U}(C_1)\cup \mathcal{O}(C_2) \cup \mathcal{U}(C_3)\ldots \cup \mathcal{U}(C_n) , &~\text{if ~} n \text{~is odd}. 
     \end{cases}
\]
\end{definition}

\vspace{0.2cm}
\noindent We now define the index values of a crossing in an oriented twisted knot diagram. Let $D$ be an oriented twisted knot diagram and let $c^*$ be a crossing of $D$.  
To compute the index values of crossing $c^*$, we first mark points $\alpha$ and $\beta$ near the under crossing and over crossing of $c^*$, respectively, as shown in Fig.~\ref{fig:Fig1}(a). Next, we  perform a smoothing operation at $c^*$ along the orientation as shown in Fig.~\ref{fig:Fig1}(b), which yields a two component twisted link diagram $D_{c^*}$. Let $C$ and $C'$ be the arcs of $D_{c^*}$ that contain the base points $\alpha$ and $\beta$, respectively. 
Assume that the component containing $\alpha$ consists of $n$ arcs, while the component containing $\beta$ consists of $m$ arcs.
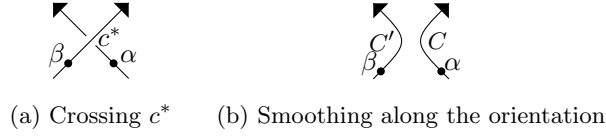
\begin{figure}
    \centering
    \begin{subfigure}{}
        \begin{tikzpicture}
        \draw(0,1)--(1,0);
        \draw[white, line width=5pt](0,0)--(1,1);
            \draw (0,0)--(1,1);
            \fill (0,0.8) -- (0.2,1) --(0,1)--cycle;
            \fill (1,1) -- (0.8,1) -- (1,0.8) -- cycle;
            \fill (0.2,0.2) circle (1.5pt);
            \fill (0.8,0.2) circle (1.5pt);
            \node at (0.05,0.3) {$\beta$};
            \node at (1,0.3) {$\alpha$};
            \node at (0.75,0.55) {$c^*$};
            \node at (0.5,-0.5) {\small (a) Crossing $c^*$};
        \end{tikzpicture}
    \end{subfigure}
    \begin{subfigure}{}
        \begin{tikzpicture}
            \draw (0,0) .. controls (0.5,0.5) .. (0,1);
            \draw (1,0) .. controls (0.5,0.5) .. (1,1);
            \fill (0,0.8) -- (0.2,1) --(0,1)--cycle;
            \fill (1,1) -- (0.8,1) -- (1,0.8) -- cycle;
            \fill (0.1,0.1) circle (1.5pt);
            \fill (0.9,0.1) circle (1.5pt);
            \node at (-0.05,0.2) {$\beta$};
            \node at (1.05,0.2) {$\alpha$};
            \node at (0.85,0.5) {\small $C$};
            \node at (0.12,0.5) {\small $C'$};
            \node at (0.5,-0.5) {\small (b) Smoothing along the orientation};
        \end{tikzpicture}
    \end{subfigure}
    \caption{Crossing $c^*$ of twisted knot diagram $D$ and twisted link diagram $D_{c^*}$.}
    \label{fig:Fig1}
\end{figure}
\\ 

\noindent To compute the index value of $c^*$ with respect to the base point $\alpha$, we consider the sequence of all consecutive arcs starting from $C$, denoted by $\langle C=C_1,C_2,\ldots, C_n \rangle $.
Then the index value  of $c^*$ with respect to the base point $\alpha$, denoted by $\operatorname{Ind}^1(c^*)$, is defined as
\[\operatorname{Ind}^1(c^*)=\displaystyle \sum_{c\in \mathcal{O}_{L}(D_{c^*}:C) }\operatorname{sgn}(c)- \displaystyle \sum_{c\in \mathcal{U}_{L}(D_{c^*}: C) }\operatorname{sgn}(c).\]   

\noindent Similarly, we consider the sequence of all consecutive arcs starting from $C'$, denoted by $\langle C'=C'_1,C'_2,\ldots, C'_m \rangle$.
Then the index value  of $c^*$ with respect to the base $\beta$, denoted by $\operatorname{Ind}^2(c^*)$, is defined as
\[\operatorname{Ind}^2(c^*)=\displaystyle \sum_{c\in \mathcal{O}_{L}(D_{c^*}: C') }\operatorname{sgn}(c)- \displaystyle \sum_{c\in \mathcal{U}_{L}(D_{c^*}: C') }\operatorname{sgn}(c).\]   

\begin{figure}[htbp]
    \centering
    \begin{subfigure}{}
    \begin{tikzpicture}[scale = 0.6]
        \draw (0,8)--(1,7);
        \draw[white, line width=6pt] (1,8)--(0,7);
        \draw (1,8)--(0,7);
        \draw (0,7) -- (1,6);
        \draw[white, line width=6pt] (1,7) -- (0,6);
        \draw (1,7) -- (0,6);
        \draw (0,6) -- (1,5);
        \draw[white, line width=6pt] (1,6) -- (0,5);
        \draw (1,6) -- (0,5);
        \draw (0,5) -- (1,4);
         \draw (0.5,4.5) circle (0.2cm); 
        \draw (1,5) -- (0,4);
        \draw (0,4) -- (1,3);
        \draw[white, line width=6pt] (1,4) -- (0,3);
        \draw (1,4) -- (0,3);
        \draw (0,3) -- (1,2);
        \draw (1,3) -- (0,2);
        \draw (0,2) -- (1,1);
        \draw[white, line width=6pt] (1,2) -- (0,1);
        \draw (1,2) -- (0,1);
        \draw (0,8) to[out=135 , in=235 ] (0,1);
        \draw(1,8) to[out=45 , in=315 ] (1,1);
        \draw(0.5,2.5) circle (0.2cm);
        \draw (0.8,6)--(1.2,6);
        \draw (0.8,4)--(1.2,4);
        \draw (-1.5,4)--(-1.1,4);
        \fill (0,1.8) -- (0.2,2.05) -- (0,2) -- cycle;
       \fill (1,1.8) -- (0.8,2.05) -- (1,2) -- cycle;
       \node at (0.85,1.5) {\footnotesize $c_1^*$};
        \node at (0.85,3.5) {\footnotesize $c_2^*$};
          \node at (0.85,5.5) {\footnotesize $c_3^*$};
           \node at (0.85,6.5) {\footnotesize $c_4^*$};
            \node at (0.85,7.55) {\footnotesize $c_5^*$} ;
            \draw[fill] (0,1) circle (0.06cm);
            \draw[fill] (1,1) circle (0.06cm);
            \node at (0,0.7) {\small $\beta$};
            \node at (1,0.7) {\small $\alpha$};
            \node at (0.5,0) {$D$};
    \end{tikzpicture}
    \end{subfigure}
    \begin{subfigure}{}
        \begin{tikzpicture}[scale=0.6]
        \draw (0,8)--(1,7);
        \draw[white, line width=6pt] (1,8)--(0,7);
        \draw (1,8)--(0,7);
        \draw (0,7) -- (1,6);
        \draw[white, line width=6pt] (1,7) -- (0,6);
        \draw (1,7) -- (0,6);
        \draw (0,6) -- (1,5);
        \draw[white, line width=6pt] (1,6) -- (0,5);
        \draw (1,6) -- (0,5);
        \draw (0,5) -- (1,4);
        \draw (0.5,4.5) circle (0.2cm);
        \draw (1,5) -- (0,4);
        \draw (0,4) -- (1,3);
        \draw[white, line width=6pt] (1,4) -- (0,3);
        \draw (1,4) -- (0,3);
        \draw (0,3) -- (1,2);
        \draw (1,3) -- (0,2);
        \draw (0,2) to[out=235 , in=45] (0,1);
        \draw (1,2) to[out=-45 , in=135] (1,1);
        \draw (0,8) to[out=135 , in=235 ] (0,1);
        \draw(1,8) to[out=45 , in=315 ] (1,1);
        \draw(0.5,2.5) circle (0.2cm);
        \draw (0.8,6)--(1.2,6);
        \draw (0.8,4)--(1.2,4);
        \draw (-1.5,4)--(-1.1,4);
       \fill (-0.15,1.8) -- (0.05,1.8) -- (-0.02,2) -- cycle;
        \fill (0.9,1.8) -- (1.15,1.8) -- (1.02,2) -- cycle;
         \node at (0.85,3.5) {\footnotesize $c_2^*$};
          \node at (0.85,5.5) {\footnotesize $c_3^*$};
           \node at (0.85,6.5) {\footnotesize $c_4^*$};
            \node at (0.85,7.55) {\footnotesize $c_5^*$} ;
            \draw[fill] (0,1) circle (0.06cm);
            \draw[fill] (1,1) circle (0.06cm);
            \node at (0,0.7) {\small $\beta$};
            \node at (1,0.7) {\small $\alpha$};
             \node at (0.5,0) {$D_{c_1^*}$};
             \node at (2.8,4) {\small \textbf{$C_1$}};
             \node at (-0.4,5) {\small \textbf{$C_2$}};
             \node at (-1.7,5) {\small \textbf{$C_3$}};
    \end{tikzpicture}
    \end{subfigure}
    \caption{Twisted knot diagrams $D$ and $D_{c_1^*}$.}
    \label{fig:exp1}
\end{figure}
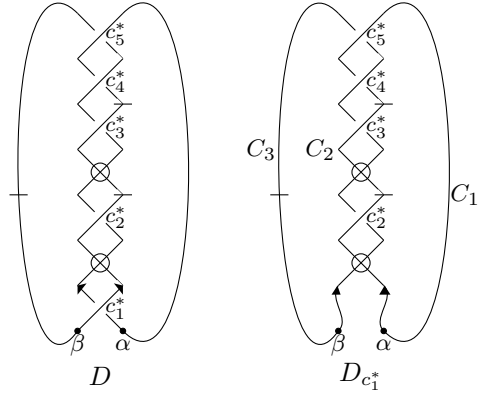

\begin{example}
Consider a twisted knot diagram $D$ with a crossing $c_1^*$, as shown in Fig.~\ref{fig:exp1}. Let $D_{c_1^*}$ denote the twisted link diagram  obtained from $D$ by smoothing the crossing $c_1^*$ along the orientation with base points $\alpha$ and $\beta$. In the diagram $D_{c_1^*}$, let $C_1, C_2$, and $C_3$ represent the arcs, as shown in Fig.~\ref{fig:exp1}. Since the arcs $C_1$ and $C_3$ contain the base points $\alpha$ and $\beta$, respectively, it follows that  $\langle C_1, C_2 \rangle$ and $\langle C_3 \rangle$ are the sequences of all consecutive arcs starting from $C_1$ and $C_3$, respectively. Then 
\[\mathcal{ O}(C_1)= \{c_5^*,c_2^*\}, \qquad \mathcal{U}(C_1)= \{c_4^*\},~~\qquad     \mathcal{O}(C_2)= \{c_3^*\}, \]
\[\quad \mathcal{U}(C_2)= \emptyset, ~\quad \mathcal{O}(C_3)= \{c_4^*\}, \quad~\mathcal{ U}(C_3)= \{c_2^*,c_3^*,c_5^*\}, \quad \text{and}\]
\[ \mathcal{O}_{L}(D_{c_1^*}:C_1) =\mathcal{ O}(C_1)~\cup~ \mathcal{ U}(C_2)= \{c_5^*,c_2^*\},~~\mathcal{U}_{L}(D_{c_1^*}:C_1) = \mathcal{U}(C_1)~\cup~ \mathcal{ O}(C_2)=\{c_3^*,c_4^*\},\]
\[ \mathcal{O}_{L}(D_{c_1^*}:C_3) =\mathcal{ O}(C_3) =\{c_4^*\},\quad \text{and}\quad  \mathcal{U}_{L}(D_{c_1^*}:C_3) =\mathcal{U}(C_3)=\{c_2^*,c_3^*,c_5^*\}.\]
Moreover, the signs of all the crossings in $D_{c_1^*}$ is positive. Hence,
\[\displaystyle \sum_{c\in \mathcal{O}_{L}(D_{c_1^*}: C_1) }\operatorname{sgn}(c)=\operatorname{sgn}(c_5^*)~+~\operatorname{sgn}(c_2^*)=2,\] 
\[\displaystyle \sum_{c\in \mathcal{U}_{L}(D_{c_1^*}: C_1) }\operatorname{sgn}(c)=\operatorname{sgn}(c_4^*)~+~\operatorname{sgn}(c_3^*)=2,\]

\[\displaystyle \sum_{c\in \mathcal{O}_{L}(D_{c_1^*}: C_3) }\operatorname{sgn}(c)=\operatorname{sgn}(c_4^*)=1, \]
\[ \displaystyle \sum_{c\in \mathcal{U}_{L}(D_{c_1^*}: C_3) }\operatorname{sgn}(c)=\operatorname{sgn}(c_2^*)+\operatorname{sgn}(c_3^*)~+~\operatorname{sgn}(c_5^*)=3.\]
Thus, the index values of the crossing $c_1^*$ are given by 
\[\operatorname{Ind}^1(c_1^*)=\displaystyle \sum_{c\in \mathcal{O}_{L}(D_{c_1^*}: C_1) }\operatorname{sgn}(c)- \displaystyle \sum_{c\in \mathcal{U}_{L}(D_{c_1^*}: C_1) }\operatorname{sgn}(c)=2-2=0,\]
\[\operatorname{Ind}^2(c_1^*)=\displaystyle \sum_{c\in \mathcal{O}_{L}(D_{c_1^*}: C_3) }\operatorname{sgn}(c)- \displaystyle \sum_{c\in \mathcal{U}_{L}(D_{c_1^*}: C_3) }\operatorname{sgn}(c)=1-3=-2.\]

\end{example}
\begin{remark}
In the case of a virtual knot diagram $D$, there are no bars in $D$. Therefore, $D$ consists of a single arc. It follows from the definition that
\[\operatorname{Ind}^1(c)=\operatorname{Ind}(c) \quad \text{and} \quad \operatorname{Ind}^2(c)=-\operatorname{Ind}(c),\]
where $\operatorname{Ind}(c)$ is the index of $c$ defined in \cite{cheng2013polynomial}.
\end{remark}
\begin{lemma}\label{prop: RI,RII}
 Let $D$ and $D'$ be two twisted knot diagrams that differ by an RI or  RII move, with $D'$ having more crossings than $D$. Then
\begin{enumerate}
\item[(a)] If $D$ and $D'$ differ by an RI move, and $c'$ is the new crossing in $D'$, then 
\[\operatorname{Ind}^1(c')=\operatorname{Ind}^2(c')=0.\]
\item[(b)] If $D$ and $D'$  differ by an RII move, and $c_1$ and $c_2$ are the new crossings in $D'$, then
\[\operatorname{Ind}^1(c_1)=\operatorname{Ind}^1(c_2)\quad \text{and} \quad \operatorname{Ind}^2(c_1)=\operatorname{Ind}^2(c_2).\]
\end{enumerate}
\end{lemma}
\begin{proof}
Consider two twisted knot diagrams $D$ and $D'$ that differ by either an RI or RII move.
\begin{itemize}
    \item[(a)] Let $D'$ be obtained from $D$ by an RI move, and let $c'$ denote the new crossing in $D'$. Depending on the orientation and type of the crossing at $c'$, all possible cases are illustrated in Fig.~\ref{fig:RI}. After smoothing the crossing $c'$ along the orientation, as shown in Fig.~\ref{fig:RI}, the resulting link diagram $D'_{c'}$ contains no linking crossing. Therefore, by definition
\[\operatorname{Ind}^1(c')=\operatorname{Ind}^2(c')=0.\] 

\begin{figure}[htbp]
    \centering
\begin{tcolorbox}[
  colframe=black,
  colback=white,
  boxrule=0.8pt,
  sharp corners
]

\begin{center}
\begin{minipage}{0.45\textwidth}
\begin{tcolorbox}[colframe=black, colback=white]
\centering
\begin{tikzpicture}[scale=0.72]
\draw (-1.5,-0.5) -- (-1.5,0.5);
\draw(-1.15,0)--(-0.65,0);
\draw (-0.8,0.1)--(-0.65,0);
\draw (-0.8,-0.1)--(-0.65,0);
\node at (-1.5,0) [
  fill=black,
  regular polygon,
  regular polygon sides=3,
  rotate=0,
   scale=0.3
] {};
    \draw (-0.5,0.5) -- (0.5,-0.5);
    \draw[white, line width=5pt] (-0.5,-0.5) -- (0.5,0.5);
     \draw (-0.5,-0.5) -- (0.5,0.5);
     \draw (0.5,0.5) .. controls (1,0) .. (0.5,-0.5);
       \node at (0.3,0.3) [
  fill=black,
  regular polygon,
  regular polygon sides=3,
  rotate=70,
   scale=0.3
] {};
\node at (-0.3,0.3) [
  fill=black,
  regular polygon,
  regular polygon sides=3,
  rotate=50,
   scale=0.3
] {};
     \draw (1.5,0) -- (2,0);
     \draw (1.9,0.1) -- (2,0);
     \draw (1.9,-0.1) -- (2,0);
     \node at (-1,0.2) {\tiny $RI$ };
     \draw (2.5,0.5) .. controls (3,0) .. (2.5,-0.5);
     \draw (4,0) circle (15pt);
      \node at (2.86,0) [
  fill=black,
  regular polygon,
  regular polygon sides=3,
  rotate=0,
   scale=0.3
] {};
      \node at (3.5,0) [
  fill=black,
  regular polygon,
  regular polygon sides=3,
  rotate=0,
   scale=0.3
] {};
      \fill (2.8,-0.2) circle (2pt) node[ left ] {\footnotesize $\beta$};
      \fill (3.55,-0.3) circle (2pt) node[right ] {\footnotesize $\alpha$};
      \node at (0,-0.75) {\footnotesize $c'$};
       
\end{tikzpicture}

Case (i)
\end{tcolorbox}
\end{minipage}
\hfill
\begin{minipage}{0.45\textwidth}
\begin{tcolorbox}[colframe=black, colback=white]
\centering
 \begin{tikzpicture}[scale=0.72]
 \draw (-1.5,-0.5) -- (-1.5,0.5);
\draw(-1.15,0)--(-0.65,0);
\draw (-0.8,0.1)--(-0.65,0);
\draw (-0.8,-0.1)--(-0.65,0);
\node at (-1.5,0) [
  fill=black,
  regular polygon,
  regular polygon sides=3,
  rotate=60,
   scale=0.3
] {};
    \draw (-0.5,0.5) -- (0.5,-0.5);
    \draw[white, line width=5pt] (-0.5,-0.5) -- (0.5,0.5);
     \draw (-0.5,-0.5) -- (0.5,0.5);
     \draw (0.5,0.5) .. controls (1,0) .. (0.5,-0.5);
     \node at (0.3,0.3) [
  fill=black,
  regular polygon,
  regular polygon sides=3,
  rotate=28,
   scale=0.3
] {};
\node at (-0.3,0.3) [
  fill=black,
  regular polygon,
  regular polygon sides=3,
  rotate=-10,
   scale=0.3
] {};
     \draw (1.5,0) -- (2,0);
     \draw (1.9,0.1) -- (2,0);
     \draw (1.9,-0.1) -- (2,0);
     \node at (-1,0.2) {\tiny $RI$ };
     \draw (2.5,0.5) .. controls (3,0) .. (2.5,-0.5);
     \draw (4,0) circle (15pt);
      \node at (2.86,0) [
  fill=black,
  regular polygon,
  regular polygon sides=3,
  rotate=-180,
   scale=0.3
] {};
      \node at (3.5,0) [
  fill=black,
  regular polygon,
  regular polygon sides=3,
  rotate=-180,
   scale=0.3
] {};
      \fill (2.8,0.2) circle (2pt) node[left ] {\footnotesize $\alpha$};
      \fill (3.55,0.3) circle (2pt) node[right ] {\footnotesize $\beta$};
      \node at (0,-0.75) {\footnotesize $c'$};
\end{tikzpicture}

Case (ii)
\end{tcolorbox}
\end{minipage}

\vspace{0.5cm}

\begin{minipage}{0.45\textwidth}
\begin{tcolorbox}[colframe=black, colback=white]
\centering
\begin{tikzpicture}[scale=0.72]
\draw (-1.5,-0.5) -- (-1.5,0.5);
\draw(-1.15,0)--(-0.65,0);
\draw (-0.8,0.1)--(-0.65,0);
\draw (-0.8,-0.1)--(-0.65,0);
\node at (-1.5,0) [
  fill=black,
  regular polygon,
  regular polygon sides=3,
  rotate=0,
   scale=0.3
] {};
   \draw (-0.5,-0.5) -- (0.5,0.5);
    \draw[white, line width=5pt] (-0.5,0.5) -- (0.5,-0.5);
    \draw (-0.5,0.5) -- (0.5,-0.5);
     \draw (0.5,0.5) .. controls (1,0) .. (0.5,-0.5);
      \node at (0.3,0.3) [
  fill=black,
  regular polygon,
  regular polygon sides=3,
  rotate=70,
   scale=0.3
] {};
\node at (-0.3,0.3) [
  fill=black,
  regular polygon,
  regular polygon sides=3,
  rotate=50,
   scale=0.3
] {};
    
     \draw (1.5,0) -- (2,0);
     \draw (1.9,0.1) -- (2,0);
     \draw (1.9,-0.1) -- (2,0);
     \node at (-1,0.2) {\tiny $RI$ };
     \draw (2.5,0.5) .. controls (3,0) .. (2.5,-0.5);
     \draw (4,0) circle (15pt);
     \node at (2.86,0) [
  fill=black,
  regular polygon,
  regular polygon sides=3,
  rotate=0,
   scale=0.3
] {};
      \node at (3.5,0) [
  fill=black,
  regular polygon,
  regular polygon sides=3,
  rotate=0,
   scale=0.3
] {};
      \fill (2.8,-0.2) circle (2pt) node[ left] {\footnotesize $\alpha$};
      \fill (3.55,-0.3) circle (2pt) node[above right] {\footnotesize $\beta$};
      \node at (0,-0.75) {\footnotesize $c'$};
     
\end{tikzpicture}

Case (iii)
\end{tcolorbox}
\end{minipage}
\hfill
\begin{minipage}{0.45\textwidth}
\begin{tcolorbox}[colframe=black, colback=white]
\centering
 \begin{tikzpicture}[scale=0.72]
  \draw (-1.5,-0.5) -- (-1.5,0.5);
\draw(-1.15,0)--(-0.65,0);
\draw (-0.8,0.1)--(-0.65,0);
\draw (-0.8,-0.1)--(-0.65,0);
\node at (-1.5,0) [
  fill=black,
  regular polygon,
  regular polygon sides=3,
  rotate=60,
   scale=0.3
] {};
     \draw (-0.5,-0.5) -- (0.5,0.5);
    \draw[white, line width=5pt] (-0.5,0.5) -- (0.5,-0.5);
      \draw (-0.5,0.5) -- (0.5,-0.5);
     \draw (0.5,0.5) .. controls (1,0) .. (0.5,-0.5);
    \node at (0.3,0.3) [
  fill=black,
  regular polygon,
  regular polygon sides=3,
  rotate=28,
   scale=0.3
] {};
\node at (-0.3,0.3) [
  fill=black,
  regular polygon,
  regular polygon sides=3,
  rotate=-10,
   scale=0.3
] {};
     \draw (1.5,0) -- (2,0);
     \draw (1.9,0.1) -- (2,0);
     \draw (1.9,-0.1) -- (2,0);
     \node at (-1,0.2) {\tiny $RI$ };
     \draw (2.5,0.5) .. controls (3,0) .. (2.5,-0.5);
     \draw (4,0) circle (15pt);
       \node at (2.86,0) [
  fill=black,
  regular polygon,
  regular polygon sides=3,
  rotate=-180,
   scale=0.3
] {};
      \node at (3.5,0) [
  fill=black,
  regular polygon,
  regular polygon sides=3,
  rotate=-180,
   scale=0.3
] {};
      \fill (2.8,0.2) circle (2pt) node[left ] {\footnotesize $\beta$};
      \fill (3.55,0.3) circle (2pt) node[right ] {\footnotesize $\alpha$};
      \node at (0,-0.75) {\footnotesize $c'$};
\end{tikzpicture}

Case (iv)
\end{tcolorbox}
\end{minipage}

\end{center}

\end{tcolorbox}
\caption{RI move and smoothing along the orientation.}
    \label{fig:RI}
\end{figure}
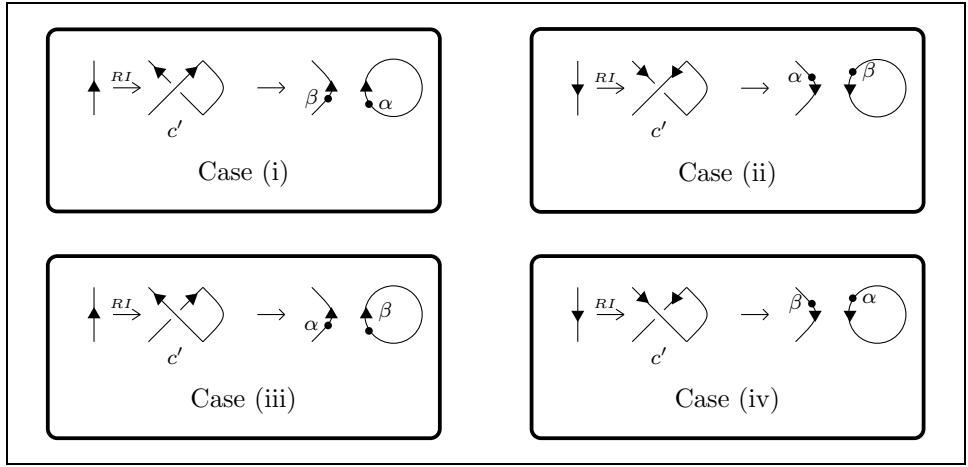

\item[(b)] \noindent Now, consider the case of an RII move, and let $c_1$ and $c_2$ be the new crossings in $D'$. Let $D'_{c_1}$ and $D'_{c_2}$ be the diagrams obtained from $D'$ by smoothing the crossings $c_1$ and $c_2$, along the orientation.  We consider two possible configurations of the RII move, presented in Fig.~\ref{fig:RII1} and Fig.~\ref{fig:RII2}. It is evident from Fig.~\ref{fig:RII1} that, in Case 1, both the diagrams $D'_{c_1}$ and $D'_{c_2}$ are equivalent, with the base points $\alpha$ and $\beta$ preserved.  Therefore,
\[\operatorname{Ind}^1(c_1)=\operatorname{Ind}^1(c_2)\quad \text{and} \quad \operatorname{Ind}^2(c_1)=\operatorname{Ind}^2(c_2).\]


\begin{figure}[htbp]
    \centering
    \begin{tikzpicture}
    \draw (-1,1) -- (-1,-1);
    \draw (0,1) -- (0,-1); 
    \fill (-1.15,0) -- (-0.85,0) -- (-1,0.3) -- cycle;
    \fill (-0.15,0) -- (0.15,0) -- (0,-0.3) -- cycle;
    \draw (1,0) --(2,0);
    \draw (1.8,0.2) --(2,0);
    \draw (1.8,-0.2) --(2,0);
    \node at (1.4,0.3) {RII};
    \node at (-0.5,-1.5) {D};

     \draw (3.8,1) to[out=180 , in=180] (3.8,-1);
    \draw[white, line width=5pt ] (3,1) to[out=0 , in=0 ] (3,-1);
    \draw (3,1) to[out=0 , in=0 ] (3,-1);
      \fill (3.13,0) -- (3.33,0) -- (3.23,-0.3) -- cycle;
    \fill (3.47,0) -- (3.67,0) -- (3.57,0.3) -- cycle;
    \node at (3.3,-1.5) {$D'$};
    \draw (4.1,0.1) -- (4.7,0.4);
    \draw (4.1,-0.1) -- (4.7,-0.4);
    \draw (4.53,0.41) -- (4.7,0.4);
     \draw (4.6,0.27) -- (4.7,0.4);
     \draw (4.55,-0.39) -- (4.7,-0.4);
     \draw (4.6,-0.28) -- (4.7,-0.4);
     \fill (6,1.4) circle (1pt);
     \fill (6,1.8) circle (1pt);
      \fill (6,-1.4) circle (1pt);
     \fill (6,-1.8) circle (1pt);
      \fill (7.8,1.65) circle (1pt);
     \fill (7.75,1.85) circle (1pt);\
      \fill (7.8,-1.65) circle (1pt);
     \fill (7.75,-1.85) circle (1pt);
      \node at (3.35,1.1) {\small $c_2$};
      \node at (3.4,-1.1) {\small $c_1$};
     \draw (5.8,1.2) -- (6.5,0.5);
     \draw[white, line width=5pt] (5.5,0.5) -- (6.5,1.5);
     \draw (5.5,0.5) -- (6.2,1.2);
      \draw (6.42,2) ..controls (6,1.7) .. (5.5,2);
    \draw (6.2,1.2) .. controls (6,1.5).. (5.8,1.2);
    \node at (5.65,0.65) [
  fill=black,
  regular polygon,
  regular polygon sides=3,
  rotate=75,
   scale=0.3
] {};
 \node at (5.8,1.85) [
  fill=black,
  regular polygon,
  regular polygon sides=3,
  rotate=65,
   scale=0.3
] {};
    \node at (6,1.7) {\tiny $\alpha$};
    \node at (6,1.52) {\tiny $\beta$};
    \draw (6.5,1.25) -- (7,1.5);
    \draw (6.9,1.55) -- (7,1.5);
    \draw (6.96,1.4) -- (7,1.5);
    \node at (6.7,1.55) {\tiny RI};
      \draw (7.5,2) to[out=-90, in= -90] (8,2);
      \draw (7.5,1.5) to[out=90, in= 90] (8,1.5);
       \node at (7.65,1.62) [
  fill=black,
  regular polygon,
  regular polygon sides=3,
  rotate=60,
   scale=0.25
] {};
 \node at (7.6,1.865) [
  fill=black,
  regular polygon,
  regular polygon sides=3,
  rotate=55,
   scale=0.25
] {};
       \node at (7.74,1.99) {\tiny $\alpha$};
       \node at (7.74,1.51) {\tiny $\beta$};
       \node at (6,0.2) {\tiny $D'_{c_2}$};

 \draw (5.8,-1.2) -- (6.5,-0.5);
     \draw[white, line width=5pt] (5.5,-0.5) -- (6.5,-1.5);
     \draw (5.5,-0.5) -- (6.2,-1.2);
      \draw (6.42,-2) ..controls (6,-1.7) .. (5.5,-2);
    \draw (6.2,-1.2) .. controls (6,-1.5).. (5.8,-1.2);
    \node at (5.65,-0.65) [
  fill=black,
  regular polygon,
  regular polygon sides=3,
  rotate=50,
   scale=0.3
] {};
 \node at (5.8,-1.83) [
  fill=black,
  regular polygon,
  regular polygon sides=3,
  rotate=55,
   scale=0.3
] {};
    \node at (6,-1.7) {\tiny $\beta$};
    \node at (6,-1.5) {\tiny $\alpha$};
    \draw (6.5,-1.25) -- (7,-1.5);
    \draw (6.9,-1.55) -- (7,-1.5);
    \draw (6.96,-1.4) -- (7,-1.5);
    \node at (6.7,-1.55) {\tiny RI};
      \draw (7.5,-2) to[out=90, in= 90] (8,-2);
      \draw (7.5,-1.5) to[out=-90, in= -90] (8,-1.5);
       \node at (7.65,-1.625) [
  fill=black,
  regular polygon,
  regular polygon sides=3,
  rotate=190,
   scale=0.25
] {};
 \node at (7.6,-1.87) [
  fill=black,
  regular polygon,
  regular polygon sides=3,
  rotate=50,
   scale=0.25
] {};
       \node at (7.74,-1.99) {\tiny $\beta$};
       \node at (7.74,-1.51) {\tiny $\alpha$};
 \node at (6,-2.4) {\tiny $D'_{c_1}$};
\end{tikzpicture}
    \caption{RII-move and smoothing along the orientation. Case (1).}
    \label{fig:RII1}
\end{figure}
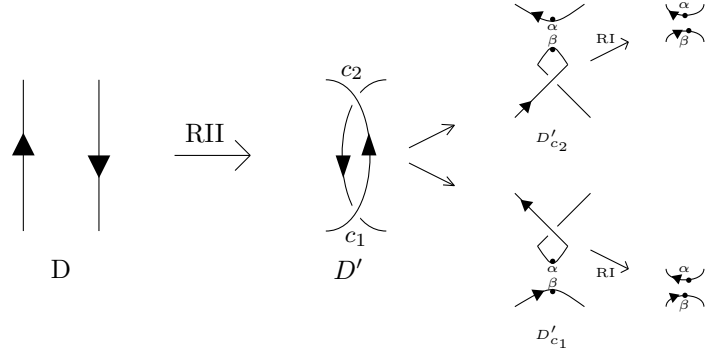
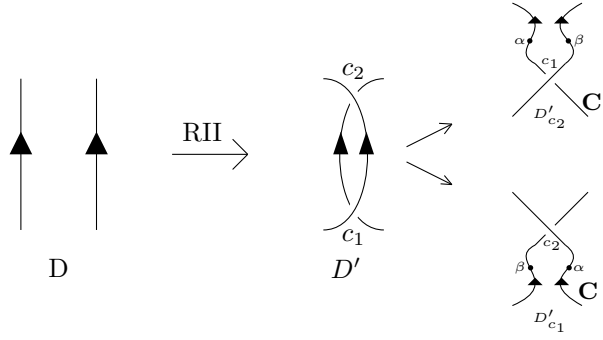
\begin{figure}[htbp]
    \centering
    \begin{tikzpicture}
    \draw (-1,1) -- (-1,-1);
    \draw (0,1) -- (0,-1); 
    \fill (-1.15,0) -- (-0.85,0) -- (-1,0.3) -- cycle;
    \fill (-0.15,0) -- (0.15,0) -- (0,0.3) -- cycle;
    \draw (1,0) --(2,0);
    \draw (1.8,0.2) --(2,0);
    \draw (1.8,-0.2) --(2,0);
    \node at (1.4,0.3) {RII};
    \node at (-0.5,-1.5) {D};

     \draw (3.8,1) to[out=180 , in=180] (3.8,-1);
    \draw[white, line width=5pt ] (3,1) to[out=0 , in=0 ] (3,-1);
    \draw (3,1) to[out=0 , in=0 ] (3,-1);
    \fill (3.13,0) -- (3.33,0) -- (3.23,0.3) -- cycle;
    \fill (3.47,0) -- (3.67,0) -- (3.57,0.3) -- cycle;
    \node at (3.4,1.1) {$c_2$};
    \node at (3.4,-1.1) {$c_1$};
    \node at (3.3,-1.5) {$D'$};
    \draw (4.1,0.1) -- (4.7,0.4);
    \draw (4.1,-0.1) -- (4.7,-0.4);
    \draw (4.53,0.41) -- (4.7,0.4);
     \draw (4.6,0.27) -- (4.7,0.4);
     \draw (4.55,-0.39) -- (4.7,-0.4);
     \draw (4.6,-0.28) -- (4.7,-0.4);

     \draw (5.8,1.2) -- (6.5,0.5);
     \draw[white, line width=5pt] (5.5,0.5) -- (6.5,1.5);
     \draw (5.5,0.5) -- (6.2,1.2);
      \draw (5.8,1.2) ..controls (5.4,1.4) and(6.2,1.7) .. (5.5,2);
    \draw (6.2,1.2) .. controls (6.6,1.5) and (5.7,1.6).. (6.42,2);
    \fill (6.05,1.73) -- (6.25,1.73) -- (6.15,1.83) -- cycle;
    \fill (5.7,1.73) -- (5.9,1.73) -- (5.8,1.83) -- cycle;
 \draw (5.8,-1.2) -- (6.5,-0.5);
    \draw[white, line width=5pt] (5.5,-0.5) -- (6.5,-1.5);
    \draw (5.5,-0.5) -- (6.2,-1.2);
     \draw (5.8,-1.2) ..controls (5.4,-1.4) and(6.2,-1.7) .. (5.5,-2);
    \draw (6.2,-1.2) .. controls (6.6,-1.5) and (5.7,-1.6).. (6.42,-2);
    \fill (6.05,-1.73) -- (6.25,-1.73) -- (6.15,-1.63) -- cycle;
    \fill (5.7,-1.73) -- (5.9,-1.73) -- (5.8,-1.63) -- cycle;
    \fill (5.73,1.5) circle (1pt);
    \fill (6.24,1.5) circle (1pt);
    \node at(5.6,1.5) {\tiny $\alpha$};
    \node at(6.38,1.5) {\tiny $\beta$};
    \node at(6,1.2) {\tiny $c_1$};
    \node at(6,0.5) {\tiny $D'_{c_2}$};
    \node at(6.55,0.7) {\small $\textbf{C}$};
    \fill (5.74,-1.5) circle (1pt);
    \fill (6.25,-1.5) circle (1pt);
    \node at(5.6,-1.5) {\tiny $\beta$};
    \node at(6.38,-1.5) {\tiny $\alpha$};
    \node at(6.025,-1.2) {\tiny $c_2$};
    \node at(6,-2.2) {\tiny $D'_{c_1}$};
    \node at(6.5,-1.8) {\small $\textbf{C}$};
\end{tikzpicture}
    \caption{RII-move and smoothing along the orientation. Case (2).}
    \label{fig:RII2}
\end{figure}
\noindent In Case 2, the diagrams $D'_{c_1}$ and $D'_{c_2}$  presented in Fig.~\ref{fig:RII2} differ at only a single crossing, although  the base points remain on the same components. Let $C$ denote the arc containing the base point $\alpha$, as shown in Fig.~\ref{fig:RII2}. It is evident that  
$c_1\in \mathcal{U}_{L}(D'_{c_2}:C)$ and $c_2\in \mathcal{O}_{L}(D'_{c_1}:C)$. Therefore,
\begin{equation}\label{pEq:2}
\operatorname{Ind}^1(c_1) - \operatorname{sgn}(c_2)=\operatorname{Ind}^1(c_2)+\operatorname{sgn}(c_1). 
\end{equation} 
Since $\operatorname{sgn}(c_1)=-\operatorname{sgn}(c_2)$, it follows from  Eq.~\ref{pEq:2} that
\[\operatorname{Ind}^1(c_1) +\operatorname{sgn}(c_1)=\operatorname{Ind}^1(c_2)+\operatorname{sgn}(c_1),  \] 
and, hence $\operatorname{Ind}^1(c_1)=\operatorname{Ind}^1(c_2)$. With similar arguments, we can prove that  $\operatorname{Ind}^2(c_1)=\operatorname{Ind}^2(c_2)$. Hence, the proof of the desired result.
\end{itemize}

\end{proof}
\begin{lemma}\label{prop: RIII}
 Let $D$  be a  twisted link diagram, and let $\langle C=C_1,C_2,\ldots, C_n \rangle$ denote the sequence of all consecutive arcs in $D$ starting from arc $C$. If $D'$ is any twisted link diagram obtained from $D$ by applying a finite number of crossing change operations, then 
 \[\displaystyle \sum_{c\in \mathcal{O}_{L}(D:C) }\operatorname{sgn}(c)- \displaystyle \sum_{c\in \mathcal{U}_{L}(D:C) }\operatorname{sgn}(c)=\displaystyle \sum_{c\in \mathcal{O}_{L}(D':C) }\operatorname{sgn}(c)- \displaystyle \sum_{c\in \mathcal{U}_{L}(D':C) }\operatorname{sgn}(c).\]
 \end{lemma}
\begin{proof} 
Consider a twisted link diagram $D$ with $m$ classical crossings. Let $\langle C=C_1,C_2,\ldots, C_n \rangle$ denote the sequence of all consecutive arcs in $D$, starting from an arc $C$. Let $D'$ be the diagram obtained from $D$ by switching $k\leq m$ crossings, among which $l$ are linking crossings, denoted by $c_1,c_2,\ldots,c_l$. Let  $c'_i$ represent the crossing in $D'$ corresponding to the crossing $c_i$ in $D$. Since $c'_i$ is obtained by switching $c_i$, we have $\operatorname{sgn}(c'_i)=-\operatorname{sgn}(c_i)$. Moreover, if 
$$  c_i\in \mathcal{O}_{L}(D:C) ~(\text{respectively},~ \mathcal{U}_{L}(D:C) ) , \text{ then } $$ 
$$ c'_i\in \mathcal{U}_{L}(D':C)~ (\text{respectively},~ \mathcal{O}_{L}(D':C)).
$$
Therefore,
 \[\displaystyle \sum_{c\in \mathcal{O}_{L}(D:C) }\operatorname{sgn}(c)- \displaystyle \sum_{c\in \mathcal{U}_{L}(D:C) }\operatorname{sgn}(c)
 =\displaystyle \sum_{c\in \mathcal{O}_{L}(D':C) }\operatorname{sgn}(c)- \displaystyle \sum_{c\in \mathcal{U}_{L}(D':C) }\operatorname{sgn}(c).\]
 Hence, the desired result follows.
\end{proof}

\noindent We now proceed to define the $\mathcal{B}$-Index polynomial.
Let $c$ be a classical crossing of an oriented twisted knot diagram $D$. Let  $D_c$ be the oriented two component link diagram obtained from $D$ by smoothing at $c$
along the orientation, as shown in Fig.~\ref{fig:Fig1}(b). The orientation of $D_c$ is induced by the orientation of the smoothing. 

\begin{definition} 
For an oriented  twisted knot diagram  $D$, we define $\mathcal{B}$-Index polynomial of $D$ as follows
\[\mathcal{B}_D(t)=\displaystyle \sum_{c\in \mathcal{C}(D)} \operatorname{sgn}(c)(t^{\operatorname{Ind^1}(c)+\operatorname{Ind^2}(c)} -1),\]
where $\mathcal{C}(D)$ denotes the set of all classical crossings of $D$.
\end{definition}

\begin{theorem} \label{thm-main} Given a diagram $D$ of a twisted knot $K$, the $\mathcal{B}$-Index polynomial $\mathcal{B}_{D}(t)$ is an invariant for $K$. 
\end{theorem}
\begin{proof} 
Let $K$ be a twisted knot, and let $D$ be a diagram representing $K$. Note that to show that $\mathcal{B}_D(t)$ is an invariant,  it is sufficient to examine the behavior of $\mathcal{B}_{D}(t)$ under the classical Reidemeister moves RI, RII, RIII, the semi virtual move SV, and the twisted TIII move.

\smallskip 

\noindent Let $D'$ be a diagram obtained from $D$ by applying one of these moves. We assume that, in the case of RI and RII the number of classical crossings in $D'$ is greater than the number of classical crossings in $D$. We now consider the following cases. 

\smallskip 

\noindent \underline{$D'$ is obtained from $D$ by RI--move:} 
Let $D'$ be a diagram obtained from $D$ by applying an RI move, and let $\overline{c}$ be the new crossing introduced in $D'$.  Let $c'$ denote the crossing in $D'$ corresponding to a crossing $c$ in $D$. It is easy to observe that  $\overline{c}$ is a self crossing in the twisted link diagram $D'_{c'}$. Consequently, the diagrams $D_{c}$ and $D'_{c'}$  also differ by an RI move, and hence the index values of $c$ and $c'$ remains the same.\\
By Lemma~\ref{prop: RI,RII}(a), we have $\operatorname{Ind}^1(\overline{c})=\operatorname{Ind}^2(\overline{c})=0,$ and hence
 \begin{align*}
\mathcal{B}_{D'}(t) &= \sum_{c\in \mathcal{C}(D')} \operatorname{sgn}(c)(t^{\operatorname{Ind}^1(c)+\operatorname{Ind}^2(c)}-1)\\
&=\sum_{c\in \mathcal{C}(D') \setminus  \{\overline{c}\}} \operatorname{sgn}(c)(t^{\operatorname{Ind}^1(c)+\operatorname{Ind}^2(c)} -1)+ \operatorname{sgn}(\overline{c})(t^{\operatorname{Ind}^1(\overline{c})+\operatorname{Ind}^2(\overline{c})}-1)\\
&=\sum_{c\in \mathcal{C}(D)} \operatorname{sgn}(c)(t^{\operatorname{Ind}^1(c)+\operatorname{Ind}^2(c)} -1)+ \operatorname{sgn}(\overline{c})(t^{\operatorname{Ind}^1(\overline{c})+\operatorname{Ind}^2(\overline{c})}-1)\\
&=\sum_{c\in \mathcal{C}(D)} \operatorname{sgn}(c)(t^{\operatorname{Ind}^1(c)+\operatorname{Ind}^2(c)} -1)+ \operatorname{sgn}(\overline{c})(t^{0}-1)=\mathcal{B}_{D}(t).
\end{align*}

\smallskip 

\noindent \underline{$D'$ is obtained from $D$ by RII--move:}
Let $D'$ be a diagram obtained from $D$ by an RII move, and let $c_1$ and $c_2$ be new crossings in $D'$. Then, 
$\operatorname{sgn}(c_1)=-\operatorname{sgn}(c_2)$, and by Lemma~\ref{prop: RI,RII}(b) 
\[\operatorname{Ind}^i(c_1) = \operatorname{Ind}^i(c_2),\quad \text{for~~} i=1,2.\]
 Let $c'$ be the crossing in $D'$ corresponding to a crossing $c$ in $D$. It is easy to observe that $D_c$ is equivalent to $D'_{c'}$ by an RII move. Moreover, both $c_1$ and $c_2$ are either self crossings or linking crossings in the link diagram $D'_{c'}$. 
\begin{enumerate}
\item If $c_1$ and $c_2$ are self crossings, then the index values of $c'$ are same as the index values of $c$. 
\item If $c_1$ and $c_2$ are linking crossings, then both  are either over linking  or under linking crossings.
 \end{enumerate} 
  Since $\operatorname{sgn}(c_1)=-\operatorname{sgn}(c_2)$, it follows that
 \[\operatorname{Ind}^i(c') = \operatorname{Ind}^i(c)+\epsilon \operatorname{sgn}(c_1)+\epsilon\operatorname{sgn}(c_2)=\operatorname{Ind}^i(c),\]
where $\epsilon=1$ or $\epsilon=-1$, depending on whether both crossings $c_1$ and $c_2$ are over linking crossings or under linking  crossings, respectively.\\

\noindent \underline{$D'$ is obtained from $D$ by RIII--move:}
Let $D'$ be a diagram obtained from $D$ by an RIII move. Let $c_1$, $c_2$, and $c_3$ be the crossings of $D$ involved in the RIII move, and let $c'_1$, $c'_2$, and $c'_3$ be the corresponding crossings in $D'$ after the RIII move, as shown in Fig.~\ref{figR3}.
It is easy to observe that the index values of any crossing $c' \neq c'_i$ in $D'$ remain the same as those of the corresponding crossing $c \neq c_i$ in $D$. Moreover, $\operatorname{sgn}(c)=\operatorname{sgn}(c')$. \\
\begin{figure}[htbp]
    \centering
    \begin{tikzpicture}
 \draw (-1,1) -- (1,-1);
    \draw[white, line width=6pt] (-1,-1) -- (1,1);
    \draw (-1,-1) -- (1,1);
    \draw[white, line width=6pt] (-2,0.5) -- (1.2,0.5);
    \draw (-1.2,0.5) -- (1.2,0.5);
       \node at (-1,1) [
  fill=black,
  regular polygon,
  regular polygon sides=3,
  rotate=45,
   scale=0.3
] {};
 \node at (1,1) [
  fill=black,
  regular polygon,
  regular polygon sides=3,
  rotate=70,
   scale=0.3
] {};
 \node at (1,0.5) [
  fill=black,
  regular polygon,
  regular polygon sides=3,
  rotate=30,
   scale=0.3
] {};
    \node at (-0.6,0.3) {\tiny $c_1$};
    \node at (0.6,0.3) {\tiny $c_2$};
    \node at (0,-0.3) {\tiny $c_3$};
    \node at (0,-1.5) {\footnotesize $D$};
    \draw (1.7,0) -- (2.3,0);
    \draw (2.3,0) -- (2.15,0.2);
    \draw (2.3,0) -- (2.15,-0.2);
    \draw (1.7,0) -- (1.85,0.2);
    \draw (1.7,0) -- (1.85,-0.2);
    \draw (3,1) -- (5,-1);
    \draw[white, line width=6pt] (3,-1) -- (5,1);
    \draw (3,-1) -- (5,1);
    \draw[white, line width=6pt] (2,-0.5) -- (5.2,-0.5);
    \draw (2.8,-0.5) -- (5.2,-0.5);
    \node at (3,1) [
  fill=black,
  regular polygon,
  regular polygon sides=3,
  rotate=45,
   scale=0.3
] {};
 \node at (5,1) [
  fill=black,
  regular polygon,
  regular polygon sides=3,
  rotate=70,
   scale=0.3
] {};
 \node at (5,-0.5) [
  fill=black,
  regular polygon,
  regular polygon sides=3,
  rotate=30,
   scale=0.3
] {};
    \node at (3.4,-0.3) {\tiny $c'_2$};
    \node at (4.6,-0.3) {\tiny $c'_1$};
    \node at (4,0.3) {\tiny $c'_3$};
     \node at (4,-1.5) {\footnotesize $D'$};
    \end{tikzpicture}
    \caption{RIII move. }
    \label{figR3}
\end{figure}
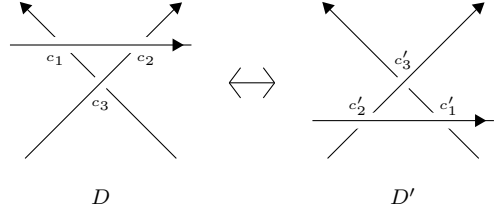

\noindent Since $\operatorname{sgn}(c_i)=\operatorname{sgn}(c'_i)$ for each $i=1,2 \text{ and } 3$, it is sufficient to examine  the behavior of index values of $c_i$ and $c'_i$.
Based on the orientation of each arc, there are eight possible cases. Here, we study one of them, as shown in Fig.~\ref{figR3}. \\
 
\noindent Let $D_{c_i}$ and $D'_{c'_i}$ be the twisted link diagrams obtained from $D$ and $D'$, respectively, by smoothing the crossings $c_i$ and $c'_i$ along the orientation. As depicted in Fig.~\ref{fig:R31} and Fig.~\ref{fig:R32}, $D_{c_i}$ and  $D'_{c'_i}$ either  represent identical diagrams or differ by a finite number of crossing changes. Moreover,  the base points on $D_{c_i}$ and  $D'_{c'_i}$ are preserved by the same components and $\operatorname{sgn} (c'_i) = \operatorname{sgn}(c_i)$. Thus, it follows from  Lemma~\ref{prop: RIII} that
\[\operatorname{Ind}^j (c'_{i}) = \operatorname{Ind}^j (c_i),\quad \text{for~}  i=1,2,3, \text{~and~} j=1,2. \]
Hence, we have  
$\mathcal{B}_{D'}(t)=\mathcal{B}_{D}(t)$.

\noindent For the remaining types of RIII-moves, the result follows by analogous arguments. \\
\begin{figure}[htbp]
    \centering
   \begin{tikzpicture}
    \draw (-1,1) -- (1,-1);
    \draw[white, line width=6pt] (-1,-1) -- (1,1);
    \draw (-1,-1) -- (1,1);
    \draw[white, line width=6pt] (-2,0.5) -- (1.2,0.5);
    \draw (-1.2,0.5) -- (1.2,0.5);
    \node at (-1,1) [
  fill=black,
  regular polygon,
  regular polygon sides=3,
  rotate=45,
   scale=0.3
] {};
 \node at (1,1) [
  fill=black,
  regular polygon,
  regular polygon sides=3,
  rotate=72,
   scale=0.3
] {};
 \node at (1,0.5) [
  fill=black,
  regular polygon,
  regular polygon sides=3,
  rotate=30,
   scale=0.3
] {}; 
    \node at (-0.6,0.3) {\tiny $c_1$};
    \node at (0.6,0.3) {\tiny $c_2$};
    \node at (0,-0.3) {\tiny $c_3$};
    \draw (1.4,0) -- (2,0);
    \draw (2,0) -- (1.8,0.2);
    \draw (2,0) -- (1.8,-0.2);
    \draw (2,0.5) .. controls (2.5,0.5) .. (2.7,1);
     \node at (2.61,0.8) [
  fill=black,
  regular polygon,
  regular polygon sides=3,
  rotate=-20,
   scale=0.3
] {};
    \draw (2.8,-1) -- (3.7,1);
    \draw[white , line width=6pt] (3,0.2) .. controls (2.7,0.2) .. (4,0.6);
    \draw (3.7,-1) .. controls (2.7,0.2) .. (4,0.6);
    \draw[white , line width=6pt] (2.8,-1) -- (3.3,0);
    \draw (2.8,-1) -- (3.265,0.05);
     \node at (3.64,0.85) [
  fill=black,
  regular polygon,
  regular polygon sides=3,
  rotate=-22,
   scale=0.3
] {};
 \node at (3.9,0.57) [
  fill=black,
  regular polygon,
  regular polygon sides=3,
  rotate=45,
   scale=0.3
] {};
     \fill (2.5,0.6) circle (2pt) node[below right] {\footnotesize $\beta$};
      \fill (3,0.1) circle (2pt) node[below left] {\footnotesize $\alpha$};
      \node at (3.2,-1.5) {\footnotesize $D_{c_1}$};
      \draw (5,1) -- (6.5,-1);
      \draw[white, line width=6pt] (4.7,0.5) .. controls (5.8,0.5) .. (6.3,1);
      \draw (4.7,0.5) .. controls (5.8,0.5) .. (6.3,1);
      \draw[white, line width=6pt] (5.75,-1) .. controls (5.8,0) .. (6.6,0.8);
      \draw (5.75,-1) .. controls (5.8,0) .. (6.6,0.8);
       \node at (5.08,0.9) [
  fill=black,
  regular polygon,
  regular polygon sides=3,
  rotate=37,
   scale=0.3
] {};
 \node at (6.2,0.9) [
  fill=black,
  regular polygon,
  regular polygon sides=3,
  rotate=65,
   scale=0.3
] {};
 \node at (6.4,0.6) [
  fill=black,
  regular polygon,
  regular polygon sides=3,
  rotate=65,
   scale=0.3
] {};
       \fill (5.8,0.6) circle (2pt) node[above left] {\footnotesize $\beta$};
      \fill (6,0.15) circle (2pt) node[below right] {\footnotesize $\alpha$};
       \node at (5.85,-1.5) {\footnotesize $D_{c_2}$};
      \draw (7.5,1) .. controls (8,0) .. (7.5,-1);
       \draw (9,1) .. controls (8.5,0) .. (9,-1);
       \draw[white, line width=6pt] (7.25,0.5) -- (9.25,0.5);
       \draw (7.25,0.5) -- (9.25,0.5);
        \node at (9.1,0.5) [
  fill=black,
  regular polygon,
  regular polygon sides=3,
  rotate=35,
   scale=0.3
] {};
 \node at (7.85,-0.15) [
  fill=black,
  regular polygon,
  regular polygon sides=3,
  rotate=120,
   scale=0.3
] {};
 \node at (8.62,-0.06) [
  fill=black,
  regular polygon,
  regular polygon sides=3,
  rotate=120,
   scale=0.3
] {};
        \fill (7.7,-0.6) circle (2pt) node[above left] {\footnotesize $\beta$};
      \fill (8.8,-0.6) circle (2pt) node[above right] {\footnotesize $\alpha$};
      \node at (8.3,-1.5) {\footnotesize $D_{c_3}$};
    \end{tikzpicture}
    \caption{Smoothing along the orientation of crossings $c_1$, $c_2$ and $c_3$.}
    \label{fig:R31}
\end{figure}
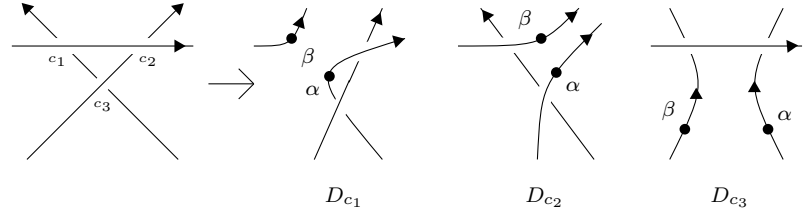
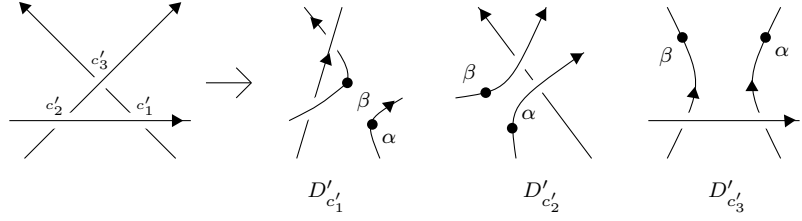
\begin{figure}[htbp]
    \centering
   \begin{tikzpicture}
    \draw (-1,1) -- (1,-1);
    \draw[white, line width=6pt] (-1,-1) -- (1,1);
    \draw (-1,-1) -- (1,1);
    \draw[white, line width=6pt] (-2,-0.5) -- (1.2,-0.5);
    \draw (-1.2,-0.5) -- (1.2,-0.5);
     \node at (-1,1) [
  fill=black,
  regular polygon,
  regular polygon sides=3,
  rotate=45,
   scale=0.3
] {};
 \node at (1,1) [
  fill=black,
  regular polygon,
  regular polygon sides=3,
  rotate=72,
   scale=0.3
] {};
 \node at (1,-0.5) [
  fill=black,
  regular polygon,
  regular polygon sides=3,
  rotate=30,
   scale=0.3
] {}; 
    \node at (-0.6,-0.3) {\tiny $c'_2$};
    \node at (0.6,-0.3) {\tiny $c'_1$};
    \node at (0,0.3) {\tiny $c'_3$};
    \draw (1.4,0) -- (2,0);
    \draw (2,0) -- (1.8,0.2);
    \draw (2,0) -- (1.8,-0.2);
    \draw (3.7,-1) .. controls (3.5,-0.5) .. (4,-0.2);
    \draw (2.5,-0.5) .. controls (3.5,0) .. (2.7,1);
    \draw[white , line width=6pt] (2.6,-1) -- (3.2,1);
    \draw (2.6,-1) -- (3.2,1);
    \draw[white , line width=6pt] (2.5,-0.5) .. controls (2.9,-0.3) ..(3.27,0);
    \draw (2.5,-0.5) .. controls (2.9,-0.3) ..(3.27,0);
      \fill (3.27,0) circle (2pt) node[below right] {\footnotesize $\beta$};
      \fill (3.6,-0.55) circle (2pt) node[below right] {\footnotesize $\alpha$};
 \node at (2.85,0.8) [
  fill=black,
  regular polygon,
  regular polygon sides=3,
  rotate=45,
   scale=0.3
] {};
 \node at (3,0.3) [
  fill=black,
  regular polygon,
  regular polygon sides=3,
  rotate=100,
   scale=0.3
] {};
 \node at (3.82,-0.3) [
  fill=black,
  regular polygon,
  regular polygon sides=3,
  rotate=70,
   scale=0.3
] {}; 
      \node at (3,-1.5) {\footnotesize $D'_{c'_1}$};
       
      \draw (5,1) -- (6.5,-1);
      \draw[white, line width=6pt] (4.7,-0.2) .. controls (5.4,-0.1) .. (5.9,1);
      \draw (4.7,-0.2) .. controls (5.4,-0.1) .. (5.9,1);
      \draw[white, line width=6pt] (5.5,-1) .. controls (5.4,-0.3) .. (6.4,0.4);
      \draw (5.5,-1) .. controls (5.4,-0.3) .. (6.4,0.4);
      \node at (5.08,0.9) [
  fill=black,
  regular polygon,
  regular polygon sides=3,
  rotate=35,
   scale=0.3
] {};
 \node at (5.85,0.9) [
  fill=black,
  regular polygon,
  regular polygon sides=3,
  rotate=90,
   scale=0.3
] {};
 \node at (6.3,0.34) [
  fill=black,
  regular polygon,
  regular polygon sides=3,
  rotate=60,
   scale=0.3
] {}; 
       \fill (5.1,-0.13) circle (2pt) node[above left] {\footnotesize $\beta$};
      \fill (5.45,-0.6) circle (2pt) node[above right] {\footnotesize $\alpha$};
      \node at (5.85,-1.5) {\footnotesize $D'_{c'_2}$};
      \draw (7.5,1) .. controls (8,0) .. (7.5,-1);
       \draw (9,1) .. controls (8.5,0) .. (9,-1);
       \draw[white, line width=6pt] (7.25,-0.5) -- (9.25,-0.5);
       \draw (7.25,-0.5) -- (9.25,-0.5);
          \node at (9.1,-0.5) [
  fill=black,
  regular polygon,
  regular polygon sides=3,
  rotate=35,
   scale=0.3
] {};
 \node at (7.85,-0.15) [
  fill=black,
  regular polygon,
  regular polygon sides=3,
  rotate=110,
   scale=0.3
] {};
 \node at (8.62,-0.06) [
  fill=black,
  regular polygon,
  regular polygon sides=3,
  rotate=120,
   scale=0.3
] {};
        \fill (7.7,0.6) circle (2pt) node[below left] {\footnotesize $\beta$};
      \fill (8.8,0.6) circle (2pt) node[below right] {\footnotesize $\alpha$};
       \node at (8.3,-1.5) {\footnotesize $D'_{c'_3}$};
    \end{tikzpicture}
    \caption{Smoothing along the orientation of crossings $c'_1$, $c'_2$ and $c'_3$.}
    \label{fig:R32}
\end{figure}


\noindent \underline{$D'$ is obtained from $D$ by SV--move:} 
Let $D'$ be a diagram obtained from $D$ by applying an SV-move at a classical crossing $c$, and let $c'$ be the corresponding crossing of $D'$. There are four possible oriented SV-moves, depending upon the orientation of the arcs. Here, we consider  two of these possibilities, as illustrated in Fig.~\ref{figSVcase1} and Fig.~\ref{figSVcase2}. It is evident from Fig.~\ref{figSVcase1} and Fig.~\ref{figSVcase2} that $D_c$ and $D'_{c'}$ are identical diagrams, with the base points preserved on the same components. Hence,
\[\operatorname{Ind}^1 (c)=\operatorname{Ind}^1 (c'),\quad \text{and}\quad \operatorname{Ind}^2 (c)=\operatorname{Ind}^2 (c'). \]
Further, $\operatorname{sgn} (c) = \operatorname{sgn}(c')$,  it follows that $\mathcal{B}_{D'}(t)=\mathcal{B}_{D}(t)$.\\

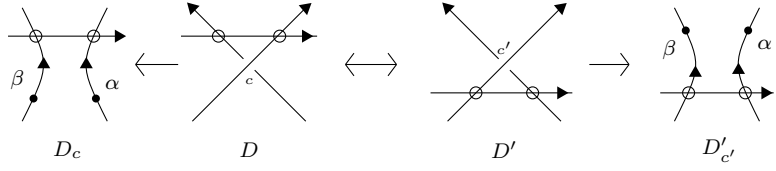
\begin{figure}[htbp]
    \centering
    \begin{tikzpicture}[scale=0.75]
       \draw (-4.25,0.5) -- (-2.25,0.5);
       \draw (-4,1) .. controls (-3.5,0) .. (-4,-1);
       \draw (-2.5,1) .. controls (-3,0) .. (-2.5,-1);
       \draw (-3.76,0.5) circle (3pt);
       \draw (-2.74,0.5) circle (3pt);
       \node at (-2.87,0) [
  fill=black,
  regular polygon,
  regular polygon sides=3,
  rotate=0,
   scale=0.3
] {};
 \node at (-3.62,0) [
  fill=black,
  regular polygon,
  regular polygon sides=3,
  rotate=0,
   scale=0.3
] {};
        \fill (-3.8,-0.6) circle (2pt) node[above left] {\footnotesize $\beta$};
      \fill (-2.7,-0.6) circle (2pt) node[above right] {\footnotesize $\alpha$};
      \node at (-3.2,-1.5) {\footnotesize $D_{c}$};
      \node at (-2.3,0.5) [
  fill=black,
  regular polygon,
  regular polygon sides=3,
  rotate=30,
   scale=0.3
] {};
      \draw (-2,0) -- (-1.25,0);
      \draw (-2,0) -- (-1.85,0.2);
      \draw (-2,0) -- (-1.85,-0.2);
       \draw (-1,1) -- (1,-1);
    \draw[white, line width=6pt] (-1,-1) -- (1,1);
    \draw (-1,-1) -- (1,1);
    \draw (-0.54,0.5) circle (3pt);
    \draw (0.54,0.5) circle (3pt);
    \draw (-1.2,0.5) -- (1.2,0.5);
    \node at (-1,1) [
  fill=black,
  regular polygon,
  regular polygon sides=3,
  rotate=45,
   scale=0.3
] {};
 \node at (1,1) [
  fill=black,
  regular polygon,
  regular polygon sides=3,
  rotate=70,
   scale=0.3
] {};
 \node at (1,0.5) [
  fill=black,
  regular polygon,
  regular polygon sides=3,
  rotate=30,
   scale=0.3
] {};
    \node at (0,-0.3) {\tiny $c$};
    \node at (0,-1.5) {\footnotesize $D$};
    \draw (1.7,0) -- (2.6,0);
    \draw (2.6,0) -- (2.45,0.2);
    \draw (2.6,0) -- (2.45,-0.2);
    \draw (1.7,0) -- (1.85,0.2);
    \draw (1.7,0) -- (1.85,-0.2);
    \draw (3.5,1) -- (5.5,-1);
    \draw[white, line width=6pt] (3.5,-1) -- (5.5,1);
    \draw (3.5,-1) -- (5.5,1);
    \draw (4,-0.5) circle (3pt);
    \draw (5,-0.5) circle (3pt);
    \draw (3.2,-0.5) -- (5.7,-0.5);
    \node at (3.5,1) [
  fill=black,
  regular polygon,
  regular polygon sides=3,
  rotate=45,
   scale=0.3
] {};
 \node at (5.5,1) [
  fill=black,
  regular polygon,
  regular polygon sides=3,
  rotate=70,
   scale=0.3
] {};
 \node at (5.5,-0.5) [
  fill=black,
  regular polygon,
  regular polygon sides=3,
  rotate=30,
   scale=0.3
] {};
    \node at (4.5,0.3) {\tiny $c'$};
     \node at (4.5,-1.5) {\footnotesize $D'$};
      \draw (6,0) -- (6.7,0);
    \draw (6.7,0) -- (6.55,0.2);
    \draw (6.7,0) -- (6.55,-0.2);

     \draw (7.5,1) .. controls (8,0) .. (7.5,-1);
       \draw (9,1) .. controls (8.5,0) .. (9,-1);
      \draw (7.75,-0.5) circle (3pt);
       \draw (8.75,-0.5) circle (3pt);
       \draw (7.25,-0.5) -- (9.25,-0.5);
       \node at (9.1,-0.5) [
  fill=black,
  regular polygon,
  regular polygon sides=3,
  rotate=30,
   scale=0.3
] {};
 \node at (7.87,-0.15) [
  fill=black,
  regular polygon,
  regular polygon sides=3,
  rotate=0,
   scale=0.3
] {};
 \node at (8.63,-0.1) [
  fill=black,
  regular polygon,
  regular polygon sides=3,
  rotate=0,
   scale=0.3
] {};
        \fill (7.7,0.6) circle (2pt) node[below left] {\footnotesize $\beta$};
      \fill (8.8,0.6) circle (2pt) node[below right] {\footnotesize $\alpha$};
       \node at (8.3,-1.5) {\footnotesize $D'_{c'}$};
\end{tikzpicture}
    \caption{SV-move and smoothing along the orientation. Case (1).}
    \label{figSVcase1}
\end{figure}
  
\begin{figure}[htbp]
    \centering
    \begin{tikzpicture}[scale=0.75]
       \draw (-4.25,0.5) -- (-2.25,0.5);
       \draw (-4,1) .. controls (-3.25,-0.5) .. (-2.5,1);
       \draw (-4,-1) .. controls (-3.25,-0.2) .. (-2.5,-1);
      \draw (-3.76,0.5) circle (3pt);
       \draw (-2.74,0.5) circle (3pt);
       \node at (-2.85,0.3) [
  fill=black,
  regular polygon,
  regular polygon sides=3,
  rotate=90,
   scale=0.3
] {};
 \node at (-2.87,-0.61) [
  fill=black,
  regular polygon,
  regular polygon sides=3,
  rotate=-20,
   scale=0.3
] {};
        \fill (-3.25,-0.44) circle (2pt) node[below] {\footnotesize $\beta$};
      \fill (-3.25,-0.12) circle (2pt) node[above] {\footnotesize $\alpha$};
      \node at (-3.2,-1.5) {\footnotesize $D_{c}$};
         \node at (-2.3,0.5) [
  fill=black,
  regular polygon,
  regular polygon sides=3,
  rotate=30,
   scale=0.3
] {};
      \draw (-2,0) -- (-1.25,0);
      \draw (-2,0) -- (-1.85,0.2);
      \draw (-2,0) -- (-1.85,-0.2);
       \draw (-1,1) -- (1,-1);
    \draw[white, line width=6pt] (-1,-1) -- (1,1);
    \draw (-1,-1) -- (1,1);
    \draw (-0.54,0.5) circle (3pt);
    \draw (0.54,0.5) circle (3pt);
    \draw (-1.2,0.5) -- (1.2,0.5);
    \node at (1,-1) [
  fill=black,
  regular polygon,
  regular polygon sides=3,
  rotate=-10,
   scale=0.3
] {};
 \node at (1,1) [
  fill=black,
  regular polygon,
  regular polygon sides=3,
  rotate=70,
   scale=0.3
] {};
 \node at (1,0.5) [
  fill=black,
  regular polygon,
  regular polygon sides=3,
  rotate=30,
   scale=0.3
] {};
    \node at (0,-0.3) {\tiny $c$};
    \node at (0,-1.5) {\footnotesize $D$};
    \draw (1.7,0) -- (2.6,0);
    \draw (2.6,0) -- (2.45,0.2);
    \draw (2.6,0) -- (2.45,-0.2);
    \draw (1.7,0) -- (1.85,0.2);
    \draw (1.7,0) -- (1.85,-0.2);
    \draw (3.5,1) -- (5.5,-1);
    \draw[white, line width=6pt] (3.5,-1) -- (5.5,1);
    \draw (3.5,-1) -- (5.5,1);
    \draw (4,-0.5) circle (3pt);
    \draw (5,-0.5) circle (3pt);
    \draw (3.2,-0.5) -- (5.7,-0.5);
      \node at (5.5,-1) [
  fill=black,
  regular polygon,
  regular polygon sides=3,
  rotate=-10,
   scale=0.3
] {};
 \node at (5.5,1) [
  fill=black,
  regular polygon,
  regular polygon sides=3,
  rotate=70,
   scale=0.3
] {};
 \node at (5.5,-0.5) [
  fill=black,
  regular polygon,
  regular polygon sides=3,
  rotate=30,
   scale=0.3
] {};
    \node at (4.5,0.3) {\tiny $c'$};
     \node at (4.5,-1.5) {\footnotesize $D'$};
      \draw (6,0) -- (6.7,0);
    \draw (6.7,0) -- (6.55,0.2);
    \draw (6.7,0) -- (6.55,-0.2);

 \draw (7.5,1) .. controls (8.25,0.5) .. (9,1);
       \draw  (7.5,-1) .. controls (8.25,0.5) .. (9,-1);
      \draw (7.75,-0.5) circle (3pt);
       \draw (8.75,-0.5) circle (3pt);
       \draw (7.25,-0.5) -- (9.25,-0.5);
         \node at (9.1,-0.5) [
  fill=black,
  regular polygon,
  regular polygon sides=3,
  rotate=30,
   scale=0.3
] {};
 \node at (7.95,-0.15) [
  fill=black,
  regular polygon,
  regular polygon sides=3,
  rotate=90,
   scale=0.3
] {};
 \node at (8.65,0.76) [
  fill=black,
  regular polygon,
  regular polygon sides=3,
  rotate=60,
   scale=0.3
] {};
        \fill (8.2,0.1) circle (2pt) node[below ] {\footnotesize $\beta$};
      \fill (8,0.7) circle (2pt) node[below ] {\footnotesize $\alpha$};
       \node at (8.3,-1.5) {\footnotesize $D'_{c'}$};
\end{tikzpicture}
    \caption{SV-move and smoothing along the orientation. Case (2).}
    \label{figSVcase2}
\end{figure}
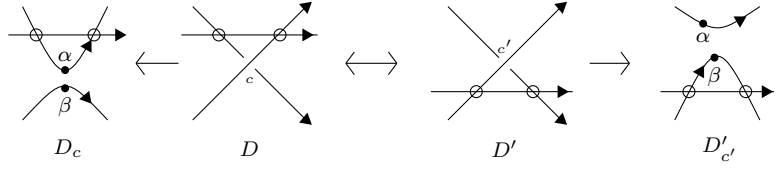

\noindent \underline{$D'$ is obtained from $D$ by TIII--move:}
\noindent Let $D'$ be the twisted knot diagram obtained from $D$ by applying a TIII-move at the crossing $\overline{c}$. Let $\overline{c}'$ denote the corresponding crossing in $D'$. Further, let $c' \neq \overline{c}'$ be a crossing in $D'$ corresponding to a crossing $c \neq \overline{c}$ in $D$.\\

\noindent We first compare the index values of the crossings $c$ and $c'$. Let $D_c$ and $D'_{c'}$ be the twisted link diagrams obtained from $D$ and $D'$, respectively, by smoothing the crossings $c$ and $c'$ according to the orientation. Let $C$ and $C'$ be the arcs containing the base point $\alpha$ (or $\beta$) in $D_c$ and $D'_{c'}$, respectively. Let $C_{i-1},C_i,C_{i+1}$ be a portion of the consecutive sequence of arcs starting from $C$ in $D_c$, and let $C'_{i-1}$ be the corresponding arc in the consecutive sequence starting from $C'$ in $D'_{c'}$. The relevant arcs are shown in Fig.~\ref{fig:TIII*}.\\

\noindent If $\overline{c}$ and $\overline{c}'$ are self-crossings in $D_c$ and $D'_{c'}$, respectively, then the local TIII-move does not affect the computation of the indices. Hence,
\[
\operatorname{Ind}^{k}(c)=\operatorname{Ind}^{k}(c'), \qquad k=1,2.
\]

\noindent Now suppose that $\overline{c}$ and $\overline{c}'$ are linking crossings in $D_c$ and $D'_{c'}$, respectively. It is easy to observe that if $c\in \mathcal{O}_L(D:C)$, then $c$ belongs to either $\mathcal{O}(C_i)$ or $\mathcal{U}(C_i)$. For instance, assume that $c\in \mathcal{O}(C_i)$. Then the corresponding crossing $c'$ satisfies $c'\in \mathcal{U}(C'_{i-1})$, and hence $c'\in \mathcal{O}_L(D':C')$. Similarly, if $c\in \mathcal{U}_L(D:C)$, then $c'\in \mathcal{U}_L(D':C')$. The remaining cases, depending on the choice of the base point, can be verified in the same manner. Therefore,
\[
\operatorname{Ind}^{k}(c)=\operatorname{Ind}^{k}(c'), \qquad k=1,2.
\]
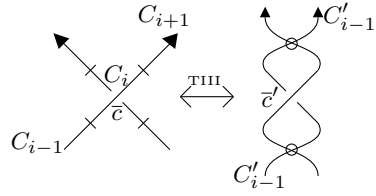
\begin{figure}[htbp]
    \centering
    \begin{subfigure}{}
        \begin{tikzpicture}[scale=0.7]
            \draw (8.2,1) -- (10.2,-1);
\draw[white, line width =6pt] (8.2,-1) -- (10.2,1);
\draw (8.2,-1) -- (10.2,1);
\draw(8.6,0.4) -- (8.8,0.6);
\draw(8.6,-0.4) -- (8.8,-0.6);
\draw(9.8,0.4) -- (9.6,0.6);
\draw(9.8,-0.4) -- (9.6,-0.6);  
\draw (10.4,0) -- (11.4,0);
\draw (10.4,0) -- (10.6,0.2);
\draw (10.4,0) -- (10.6,-0.2);
\draw (11.2,0.2) -- (11.4,0);
\draw (11.2,-0.2) -- (11.4,0);
\node at (10.85,0.3) {\tiny TIII};
\node at (9.2,-0.3) {\small $\overline{c}$};
\node at (12.1,0) {\small $\overline{c}'$};
\draw (13,1.5) to[out=-90  , in=120 ] (12,0.5);
\draw (13,0.5) to[out=60  , in=-90 ] (12,1.5);
\draw (12,0.5) -- (13,-0.5);
\draw[white, line width=6pt] (12,-0.5) -- (13,0.5);
\draw (12,-0.5) -- (13,0.5);
\draw (13,-1.5) to[out=90  , in=-120 ] (12,-0.5);
\draw (13,-0.5) to[out=-60  , in=90 ] (12,-1.5);
 \draw (12.5,1) circle (3pt);
 \draw (12.5,-1) circle (3pt);
 \fill (8.1,0.9) -- (8.3,1.1) --(8,1.2) --cycle;
 \fill (10.3,0.9) -- (10.35,1.2) --(10.05,1.1) --cycle;
 \fill (12.9,1.5) -- (13.1,1.5) -- (13,1.7) -- cycle;
  \fill (11.9,1.5) -- (12.1,1.5) -- (12,1.7) -- cycle;
 \node at (9.2,0.4) {\small $C_{i}$};
 \node at (7.7,-0.9) {\small $C_{i-1}$};
 \node at (10.1,1.5) {\small $C_{i+1}$};
 \node at (11.9,-1.5) {\small $C'_{i-1}$};
 \node at (13.6,1.5) {\small $C'_{i-1}$};
        \end{tikzpicture}
    \end{subfigure}
    \caption{TIII-move. }
    \label{fig:TIII*}
\end{figure}

\noindent It remains to compare the index values at the crossings $\overline{c}$ and $\overline{c}'$. By smoothing $\overline{c}$ and $\overline{c}'$ in the diagrams $D$ and $D'$, respectively, we obtain two equivalent diagrams, $D_{\overline{c}}$ and $D'_{\overline{c}'}$, with the base points interchanged, as shown in Fig.~\ref{fig:TIII}. Therefore,
\[
\operatorname{Ind}^{1}(\overline{c})
=
\operatorname{Ind}^{2}(\overline{c}')
\quad \text{and} \quad
\operatorname{Ind}^{2}(\overline{c})
=
\operatorname{Ind}^{1}(\overline{c}').
\]

\noindent Hence,
\[
\operatorname{Ind}^{1}(\overline{c})
+\operatorname{Ind}^{2}(\overline{c})
=
\operatorname{Ind}^{1}(\overline{c}')
+\operatorname{Ind}^{2}(\overline{c}').
\]

\noindent Therefore, every crossing contributes equally to $\mathcal{B}_{D}(t)$ and $\mathcal{B}_{D'}(t)$. Consequently,
\[
\mathcal{B}_{D}(t)=\mathcal{B}_{D'}(t).
\]
\begin{figure}[htbp]
    \centering
    \begin{subfigure}{}
        \begin{tikzpicture}[scale=0.7]
            \draw (8.2,1) -- (10.2,-1);
\draw[white, line width =6pt] (8.2,-1) -- (10.2,1);
\draw (8.2,-1) -- (10.2,1);
\draw(8.6,0.4) -- (8.8,0.6);
\draw(8.6,-0.4) -- (8.8,-0.6);
\draw(9.8,0.4) -- (9.6,0.6);
\draw(9.8,-0.4) -- (9.6,-0.6);  
\draw (10.4,0) -- (11.4,0);
\draw (10.4,0) -- (10.6,0.2);
\draw (10.4,0) -- (10.6,-0.2);
\draw (11.2,0.2) -- (11.4,0);
\draw (11.2,-0.2) -- (11.4,0);
\node at (10.85,0.3) {\tiny TIII};
\node at (8.9,0) {\small $\overline{c}$};
\node at (12.1,0) {\small $\overline{c}'$};
\draw (13,1.5) to[out=-90  , in=120 ] (12,0.5);
\draw (13,0.5) to[out=60  , in=-90 ] (12,1.5);
\draw (12,0.5) -- (13,-0.5);
\draw[white, line width=6pt] (12,-0.5) -- (13,0.5);
\draw (12,-0.5) -- (13,0.5);
\draw (13,-1.5) to[out=90  , in=-120 ] (12,-0.5);
\draw (13,-0.5) to[out=-60  , in=90 ] (12,-1.5);
 \draw (12.5,1) circle (3pt);
 \draw (12.5,-1) circle (3pt);
 \fill (8.1,0.9) -- (8.3,1.1) --(8,1.2) --cycle;
 \fill (10.3,0.9) -- (10.35,1.2) --(10.05,1.1) --cycle;
 \fill (12.9,1.5) -- (13.1,1.5) -- (13,1.7) -- cycle;
  \fill (11.9,1.5) -- (12.1,1.5) -- (12,1.7) -- cycle;
 \fill (8.8,-0.4) circle (2pt);
 \fill (9.6,-0.4) circle (2pt);
 \fill (12.9,-0.4) circle (2pt);
 \fill (12.1,-0.4) circle (2pt);
 \node at (9,-0.4) {\tiny $\beta$};
 \node at (9.38,-0.38) {\tiny $\alpha$};
 \node at (11.9,-0.4) {\tiny $\beta$};
 \node at (13.2,-0.4) {\tiny $\alpha$};
 \draw (6.8,0) -- (7.3,0);
 \draw (6.8,0) -- (6.95,0.1);
 \draw(6.8,0)--(6.95,-0.1);
 \draw(5,1).. controls (5.5,0) ..(5,-1);
 \draw(6,1).. controls (5.5,0)..(6,-1);
 \draw (5.4,0.55) --(5.1,0.5);
 \draw (5.4,-0.55) -- (5.1,-0.5);
 \draw (5.9,0.5) -- (5.6,0.55);
 \draw(5.9,-0.5) -- (5.6,-0.55);
 \fill (4.95,0.9) -- (4.95,1.1)--(5.15,1) -- cycle;
 \fill (5.85,1)--(6.1,0.95)--(6.05,1.15)--cycle;
 \fill (5.3,-0.25) circle (2pt);
 \fill (5.7,-0.25) circle (2pt);
 \node at (5.1,-0.25) {\tiny $\beta$};
 \node at (5.9,-0.25) {\tiny $\alpha$};
 \draw(14,0) -- (14.5,0);
 \draw(14.5,0) -- (14.35,0.1);
  \draw(14.5,0) -- (14.35,-0.1);
 \draw(15,0.75).. controls (15.2,0)..(15,-0.75);
 \draw(16,0.75).. controls (15.8,0).. (16,-0.75);
 \draw (15,0.75) to[out=120 , in=270 ] (16,1.5);
 \draw(16,0.75) to[out=60 , in=270 ] (15,1.5);
 \draw (15,-0.75) to[out=-120 , in=-270 ] (16,-1.5);
 \draw(16,-0.75) to[out=-60 , in=-270 ] (15,-1.5);
 \draw (15.5,1.15) circle (3.5pt);
  \draw (15.5,-1.15) circle (3.5pt);
  \fill (14.9,1.5)--(15.1,1.5)--(15,1.65)--cycle;
  \fill (15.9,1.5) -- (16.1,1.5) -- (16,1.65) --cycle;
  \fill (15.1,-0.35) circle (2pt);
   \fill (15.9,-0.35) circle (2pt);
   \node at (14.9,-0.35) {\tiny $\beta$};
   \node at (16.1,-0.35) {\tiny $\alpha$};
   \node at (9.2,-2) {$D$};
   \node at (12.5,-2) {$D'$};
   \node at (5.5,-2) {$D_{\overline{c}}$};
    \node at (15.5,-2) {$D'_{\overline{c}'}$};
     \draw(2,1).. controls (2.5,0) ..(2,-1);
 \draw(3,1).. controls (2.5,0)..(3,-1);
 \draw (4.5,0) --(4.3,0.2);
 \draw (4.5,0) --(4.3,-0.2);
 \draw (3.5,0) -- (3.7,-0.2);
 \draw (3.5,0) -- (3.7,0.2);
  \draw(3.5,0) -- (4.5,0);
 \fill (1.95,0.9) -- (1.95,1.1)--(2.15,1) -- cycle;
 \fill (2.85,1)--(3.1,0.95)--(3.05,1.15)--cycle;
 \fill (2.3,-0.25) circle (2pt);
 \fill (2.7,-0.25) circle (2pt);
 \node at (2.1,-0.25) {\tiny $\beta$};
 \node at (2.9,-0.25) {\tiny $\alpha$};
 \node at (4,0.3) {\tiny TII};
 \draw(18,1).. controls (18.5,0) ..(18,-1);
 \draw(19,1).. controls (18.5,0)..(19,-1);
 \draw (16.5,0) --(16.7,0.2);
 \draw (16.5,0) --(16.7,-0.2);
 \draw (17.5,0) -- (17.3,-0.2);
 \draw (17.5,0) -- (17.3,0.2);
  \draw(16.5,0) -- (17.5,0);
 \fill (17.95,0.9) -- (17.95,1.1)--(18.15,1) -- cycle;
 \fill (18.85,1)--(19.1,0.95)--(19.05,1.15)--cycle;
 \fill (18.3,-0.25) circle (2pt);
 \fill (18.7,-0.25) circle (2pt);
 \node at (18.1,-0.25) {\tiny $\alpha$};
 \node at (18.9,-0.25) {\tiny $\beta$};
 \node at (17,0.3) {\tiny VII};
        \end{tikzpicture}
    \end{subfigure}
    \caption{TIII-move and smoothing along the orientation at $c$ and $c'$.}
    \label{fig:TIII}
\end{figure}
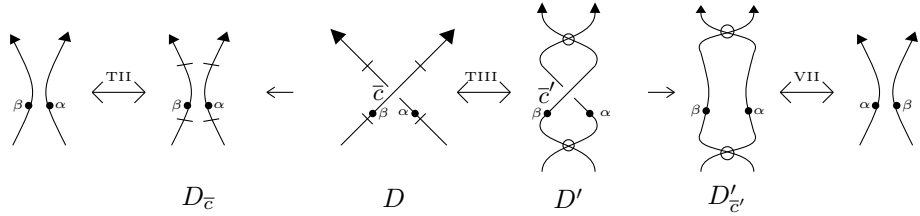
Similarly, the result follows for the other orientation of the $T$III-move.


\noindent Considering all the cases of the moves RI, RII, RIII, SV, and TIII, we conclude that $\mathcal{B}_{D}(t)$ is a twisted virtual knot invariant.
\end{proof} 

 \begin{example}\label{exp:index}
  Let $D$ be a twisted knot diagram as shown in Fig.~\ref{fig:exp1}.
  Then the indices of the crossings $c_1^*,c_2^*,c_3^*,c_4^*,c_5^*$ are given in the Table~\ref{tab:tb1}.\\
  \begin{table}[h]
\centering
\begin{tabular}{|c|c|c|c|}
\hline
Crossing($*$) & $\operatorname{Ind}^1(*)$ & $\operatorname{Ind}^2(*)$ & $\operatorname{Ind}^1(*)+\operatorname{Ind}^2(*)$ \\ \hline
 $c_1^*$ & 0 & -2 & -2  \\ \hline
 $c_2^*$ & -2 & 0 & -2 \\ \hline
 $c_3^*$ & -2 & -2 & -4 \\ \hline
 $c_4^*$ & -2 & -2 & -4 \\ \hline
 $c_5^*$ & -2 & -2 & -4 \\ \hline
\end{tabular}
\caption{Index values of the crossings of $D$ shown in Fig.~\ref{fig:exp1}. }
\label{tab:tb1}
\end{table}

\noindent Then the $\mathcal{B}$-Index polynomial of $D$ is given as
   \[\mathcal{B}_D(t)=(t^{-2}-1)+(t^{-2}-1)+3(t^{-4}-1)=2(t^{-2}-1)+3(t^{-4}-1).\]
   \end{example}
\begin{example}\label{ex:382}
   Let $D$ be the oriented twisted knot diagram shown in Fig.~\ref{fig:exmpsec2} (a). To compute the $\mathcal{B}$-Index polynomial of $D$, we first smooth each crossing $c_1$, $c_2$, and $c_3$ according to the orientation, obtaining the twisted link diagrams $D_{c_1}$, $D_{c_2}$, and $D_{c_3}$, respectively as shown in ~\ref{fig:exmpsec2}(b) - \ref{fig:exmpsec2}(d).

    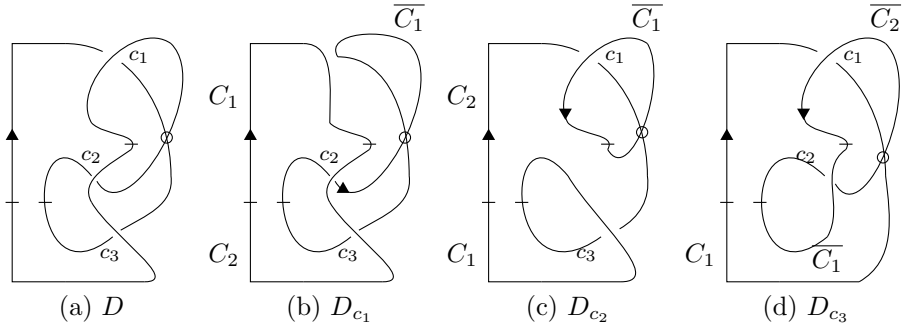
\begin{figure}[htbp]
        \centering
            \begin{tikzpicture}[scale=0.35]
            \begin{scope}
         \draw (2,0) -- (0,0) -- (0,9) -- (2,9);
         \draw (2,9) to[out=0, in =90] (6,4);
         \draw (6,4) to[out=-83, in=30] (4.1,1.95);
           \draw (3.8,1.7) .. controls (0.5,-1) and (0.5,7) .. (3,4);
           \draw[white,line width=8pt](6,9) to[out=145, in=120] (3,6);
          \draw (3.2,3.8) to[out=-60 , in=-45 ] (6,9);
          \draw (6,9) to[out=145, in=120] (3,6);
          \draw (3,6) to[out=-45 , in=60 ] (4.5,5);
          \draw (4.5,5) to[out=-145 , in=120 ] (3,3);
          \draw (3,3) to[out=-60 , in=0 ] (5,0);
          \draw (5,0) -- (2,0);
          \draw (5.84,5.45) circle (6pt);
          \draw (1,3) -- (1.5,3);
            \draw (-0.25,3) -- (0.25,3);
              \draw (4.75,5.2) -- (4.25,5.2);
              \node at (4.8,8.5) {\footnotesize $c_1$};
              \node at (3,4.7) {\footnotesize $c_2$};
              \node at (3.7,1) {\footnotesize $c_3$};
              \node at (0,5.5) [
  fill=black,
  regular polygon,
  regular polygon sides=3,
  rotate=0,
   scale=0.3
] {};
\node at (3,-1) {(a) $D$};
        \end{scope}
         \begin{scope}[xshift=9cm]
         \draw (2,0) -- (0,0) -- (0,9) -- (2,9);
         \draw (2,9) to[out=0, in =90] (3,6);
         \draw (3.3,8.5) to[out=0, in =90] (6,4);
         \draw (6,4) to[out=-83, in=30] (4.1,1.95);
           \draw (3.8,1.7) .. controls (0.5,-1) and (0.5,7) .. (3,4);
          \draw (3.2,3.8) to[out=-60 , in=-45 ] (6,9);
          \draw (6,9) to[out=145, in=120] (3.3,8.5);
          \draw (3,6) to[out=-45 , in=60 ] (4.5,5);
          \draw (4.5,5) to[out=-145 , in=120 ] (3,3);
          \draw (3,3) to[out=-60 , in=0 ] (5,0);
          \draw (5,0) -- (2,0);
          \draw (5.84,5.45) circle (6pt);
          \draw (1,3) -- (1.5,3);
            \draw (-0.25,3) -- (0.25,3);
              \draw (4.75,5.2) -- (4.25,5.2);
              \node at (3,4.7) {\footnotesize $c_2$};
              \node at (3.7,1) {\footnotesize $c_3$};
              \node at (-1,7) {$C_1$};
              \node at (-1,1) {$C_2$};
              \node at (6,10) {$\overline{C_1}$};
              \node at (0,5.5) [
  fill=black,
  regular polygon,
  regular polygon sides=3,
  rotate=0,
   scale=0.3
] {};
 \node at (3.5,3.5) [
  fill=black,
  regular polygon,
  regular polygon sides=3,
  rotate=0,
   scale=0.3
] {};
\node at (3,-1) {(b) $D_{c_1}$};
        \end{scope}
       \begin{scope}[xshift=18cm]
         \draw (2,0) -- (0,0) -- (0,9) -- (2,9);
         \draw (2,9) to[out=0, in =90] (6,4);
         \draw (6,4) to[out=-83, in=30] (4.95,2);
           \draw (4.25,1.7) .. controls (0.5,-1) and (0.5,7) .. (3,4);
           \draw[white,line width=8pt](6,9) to[out=145, in=120] (3,6);
          \draw (6,9) to[out=145, in=120] (3,6);
          \draw (3,6) to[out=-45 , in=60 ] (4.5,5);
          \draw (4.5,5) to[out=-60 , in=-60 ] (6,9);
          \draw (3,4) to[out=-60 , in=0 ] (5,0);
          \draw (5,0) -- (2,0);
          \draw (5.8,5.65) circle (6pt);
          \draw (1,3) -- (1.5,3);
            \draw (-0.25,3) -- (0.25,3);
              \draw (4.75,5.2) -- (4.25,5.2);
              \node at (4.8,8.5) {\footnotesize $c_1$};
              \node at (3.7,1) {\footnotesize $c_3$};
               \node at (-1,7) {$C_2$};
              \node at (-1,1) {$C_1$};
              \node at (6,10) {$\overline{C_1}$};
              \node at (0,5.5) [
  fill=black,
  regular polygon,
  regular polygon sides=3,
  rotate=0,
   scale=0.3
] {};
 \node at (2.9,6.4) [
  fill=black,
  regular polygon,
  regular polygon sides=3,
  rotate=180,
   scale=0.3
] {};
\node at (3,-1) {(c) $D_{c_2}$};
        \end{scope}
 \begin{scope}[xshift=27cm]
         \draw (2,0) -- (0,0) -- (0,9) -- (2,9);
         \draw (2,9) to[out=0, in =90] (6,4);
         \draw (6,4) to[out=-83, in=30] (5,0);
           \draw (3.8,1.7) .. controls (0.5,-1) and (0.5,7) .. (3.7,4);
           \draw[white,line width=8pt](6,9) to[out=145, in=120] (3,6);
          \draw (4.1,3.7) to[out=-60 , in=-45 ] (6,9);
          \draw (6,9) to[out=145, in=120] (3,6);
          \draw (3,6) to[out=-45 , in=60 ] (4.5,5);
          \draw (4.5,5) to[out=-145 , in=60 ] (3.8,1.7);
          \draw (5,0) -- (2,0);
          \draw (5.9,4.7) circle (6pt);
          \draw (1,3) -- (1.5,3);
            \draw (-0.25,3) -- (0.25,3);
              \draw (4.75,5.2) -- (4.25,5.2);
              \node at (4.8,8.5) {\footnotesize $c_1$};
              \node at (3,4.7) {\footnotesize $c_2$};
              \node at (3.8,0.9) {$\overline{C_1}$};
              \node at (-1,1) {$C_1$};
               \node at (6,10) {$\overline{C_2}$};
              \node at (0,5.5) [
  fill=black,
  regular polygon,
  regular polygon sides=3,
  rotate=0,
   scale=0.3
] {};
 \node at (2.9,6.4) [
  fill=black,
  regular polygon,
  regular polygon sides=3,
  rotate=180,
   scale=0.3
] {};
\node at (3,-1) {(d) $D_{c_3}$};
        \end{scope}

    \end{tikzpicture}
        \caption{Twisted knot diagrams $D$, $D_{c_1}$, $D_{c_2}$ and $D_{c_3}$. }
        \label{fig:exmpsec2}
    \end{figure}

\noindent For the crossing $c_1$, from the diagram $D_{c_1}$, we have
\[
\mathcal{O}(C_1)=\emptyset,\qquad \qquad~
\mathcal{U}(C_1)=\emptyset,
\]
\[
\mathcal{O}(C_2)=\{c_2,c_3\},\qquad
\mathcal{U}(C_2)=\emptyset,
\]
and
\[
\mathcal{O}(\overline{C}_1)=\emptyset,\qquad \quad
\mathcal{U}(\overline{C}_1)=\{c_2,c_3\}.
\]
Therefore,
\[
\operatorname{Ind}^1(c_1)=0,
~\text{ and } ~
\operatorname{Ind}^2(c_1)=0.
\]

\noindent Similarly, for the crossing $c_2$, the diagram $D_{c_2}$ yields
\[
\mathcal{O}(C_1)=\emptyset,\qquad \quad
\mathcal{U}(C_1)=\emptyset,
\]
\[
\mathcal{O}(C_2)=\emptyset,\qquad
\mathcal{U}(C_2)=\{c_1\},
\]
and
\[
\mathcal{O}(\overline{C}_1)=\{c_1\},\qquad
\mathcal{U}(\overline{C}_1)=\emptyset.
\]
Hence,
\[
\operatorname{Ind}^1(c_2)=1,
~\text{ and } ~
\operatorname{Ind}^2(c_2)=1.
\]

\noindent For the crossing $c_3$, from the diagram $D_{c_3}$, we obtain
\[
\mathcal{O}(C_1)=\emptyset,\qquad
\mathcal{U}(C_1)=\{c_1\},
\]
\[
\mathcal{O}(\overline{C}_1)=\emptyset,\qquad \quad
\mathcal{U}(\overline{C}_1)=\emptyset,
\]
and
\[
\mathcal{O}(\overline{C}_2)=\{c_1\},\qquad
\mathcal{U}(\overline{C}_2)=\emptyset.
\]
Thus,
\[
\operatorname{Ind}^1(c_3)=-1,
~\text{ and } ~
\operatorname{Ind}^2(c_3)=-1.
\]
Since
\[
\operatorname{sgn}(c_1)=1,\qquad
\operatorname{sgn}(c_2)=1,~\text{ and } ~
\operatorname{sgn}(c_3)=-1,
\]
it follows that
\[
\mathcal{B}_D(t)
=
\sum_{i=1}^{3}
\operatorname{sgn}(c_i)
\left(
t^{\operatorname{Ind}^1(c_i)+\operatorname{Ind}^2(c_i)}
-1
\right).
\]
Substituting the computed index values, we obtain
\[
\mathcal{B}_D(t)
=
(t^0-1)+(t^2-1)-\bigl(t^{-2}-1\bigr)
=
t^2-t^{-2}.
\]

\end{example}



\begin{remark}
In \cite{KamadaKawataOkuboShimizu}, Kamada et al.\ posed the question of whether the twisted knot diagram $3_{82}$(=$D$ in Example~\ref{ex:382}) is equivalent to the trivial non-orientable curve (a trivial loop with a bar). Let $D$ and $D^*$ denote the twisted knot diagram $3_{82}$ and the trivial non-orientable curve, respectively.

\noindent From the previous example, we have
\[
\mathcal{B}_{D}(t)=t^2-t^{-2},
\]
while a direct computation shows that
\[
\mathcal{B}_{D^*}(t)=0.
\]
Therefore,
\[
\mathcal{B}_{D}(t)\neq \mathcal{B}_{D^*}(t).
\]
Hence, the $\mathcal{B}$-Index polynomial distinguishes $D$ and $D^*$, showing that the twisted knot diagram $3_{82}$ is not equivalent to the trivial non-orientable curve. Thus, the $\mathcal{B}$-Index polynomial answers the question posed in \cite{KamadaKawataOkuboShimizu}.
\end{remark}

   \begin{remark} Consider the twisted knot diagrams $D_1$ and $D_2$ shown in Fig.~\ref{fig:D1D2}. The polynomial invariant $\bar{Q}_D$ introduced in~\cite{NKIP} does not distinguish these diagrams, since \[ \bar{Q}_{D_1}=\bar{Q}_{D_2}=-4(t-1). \] However, the $\mathcal{B}$-Index polynomial distinguishes them, \[ \mathcal{B}_{D_1}(t)=2(1-t^2) \qquad \text{and} \qquad \mathcal{B}_{D_2}(t)=0. \] Therefore, \[ \mathcal{B}_{D_1}(t)\neq \mathcal{B}_{D_2}(t), \] and, hence $D_1$ and $D_2$ are not equivalent. 
    \begin{figure}[htbp]
        \centering
        \begin{subfigure}{}
            \begin{tikzpicture}[scale=0.8]
                \draw (1,0)--(0,1);
                \draw[white, line width=4pt] (0,0)--(1,1);
                \draw (0,0)--(1,1);
                \draw (1,1)--(2,0);
                \draw[white, line width=4pt] (1,0)--(2,1);
                \draw (1,0)--(2,1);
                \draw (2,1)--(3,0);
                \draw[white, line width=4pt] (2,0)--(3,1);
                \draw (2,0)--(3,1);
                \draw (0,1) to[out=135 , in=45 ] (3,1);
                \draw (-0.5,0.5).. controls (-1,2) ..(2.5,2)..controls (4,1.9) .. (3.5,0.5);
                \draw (-0.5,0.5) to[out=-55,in=235] (0,0);
                \draw (3.5,0.5) to[out=-120,in=-55] (3,0);
                \draw (2,0.2)--(2,-0.1);
                \draw (2.9,1.1)--(3.15,0.9);
                \fill (2.5,2.15)--(2.5,1.85)--(2.65,2)--cycle;
                \node at (1.5,-0.7) {$D_1$};
            \end{tikzpicture}
        \end{subfigure}
        \begin{subfigure}{}
             \begin{tikzpicture}[scale=0.7]
                \draw (0,0)--(1,1);
                \draw[white, line width=4pt] (1,0)--(0,1);
                \draw (1,0)--(0,1);
               \draw (1,0)--(2,1);
                \draw[white, line width=4pt] (1,1)--(2,0);
                 \draw (1,1)--(2,0);
                \draw (2,0)--(3,1);
                \draw[white, line width=4pt] (2,1)--(3,0);
               \draw (2,1)--(3,0); 
               \draw (3,0) -- (4,1);
               \draw[white, line width=4pt] (3,1) -- (4,0);
               \draw (3,1) -- (4,0);
               \draw (4,1) .. controls (6,1.5) and (4.5,-1) ..(4,0);
               \draw[white, line width=4pt] (4.8,0)--(4.9,0.4);
               \draw (0,1) .. controls (-1,2) and (6,2) ..(5,1) .. controls (3.5,0.3) and (6,0.3) .. (4.8,-0.9) .. controls (0,-1) .. (0,0);
                \fill (2.5,1.9)--(2.5,1.6)--(2.65,1.75)--cycle;
                \draw (4.88,0.95) circle (4pt);
                \node at (2.5,-1.5) {$D_2$};
                \end{tikzpicture}
        \end{subfigure}
        \caption{Diagrams $D_1$ and $D_2$.}
        \label{fig:D1D2}
    \end{figure}
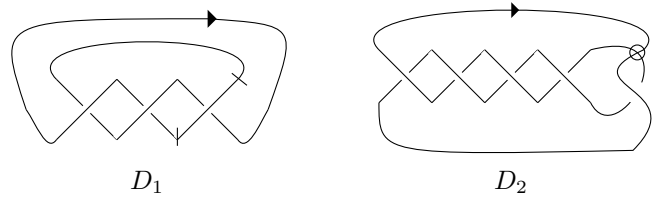
   \end{remark}
  
\noindent In the following proposition, we construct a family of twisted knots $D_n$, for $n>1$, whose polynomial is given in terms of $n$.

\begin{proposition}\label{prop:Dn}
For each integer $n>1$, there exists a twisted knot $D_n$ whose
$\mathcal{B}$-Index polynomial is given by
\[
\mathcal{B}_{D_n}(t)=2t^{4(n-1)}+2(n-1)t^4-2n.
\]
\end{proposition}

\begin{proof}
Consider the family of twisted knots $\{D_n\}_{n>1}$ shown in
Fig.~\ref{fig:placeholder2}, where $D_n$ is obtained by repeating the indicated local configuration $n-1$ times in the diagram $D_2$.

\begin{figure}[htbp]
    \centering
    \begin{subfigure}{}
         \begin{tikzpicture}[scale=0.42]
          \draw (1,0.5) -- (1,-0.5);
          \draw[white, line width=6pt] (0.9,0) -- (1.1,0);
           \draw (0,0) --(7,0);
           \draw[white, line width=6pt] (2,0.5) -- (2,-0.5);
          \draw (2,0.5) -- (2,-0.5);
          \draw (3,0.5) -- (3,-0.5);
          \draw[white, line width=6pt] (4,0.5) -- (4,-0.5);
          \draw (4,0.5) -- (4,-0.5);
          \draw (5,0.5) -- (5,-0.5);
          \draw[white, line width=6pt] (6,0.5) -- (6,-0.5);
          \draw (6,0.5) -- (6,-0.5);
         \draw (0,0) to[out=-100, in=-100 ] (6,-0.5);
          \draw (1,0.5) to[out=80 , in=100 ] (2,0.5);
        \draw (2,-0.5) to[out=-80 , in=-100 ] (3,-0.5);
         \draw (3,0.5) to[out=80 , in=100 ] (4,0.5);
        \draw (4,-0.5) to[out=-80 , in=-100 ] (5,-0.5);
         \draw (5,0.5) to[out=80 , in=100 ] (6,0.5);
          \draw (1,-0.5) to[out=-80 , in=-40 ] (7,0);
         \draw (5,-1.6) circle (0.3cm);
         \draw (5.5,-0.2) -- (5.5,0.2);
        \draw (1.5,0.2) -- (1.5,-0.2);
         \draw (1.5,1) -- (1.5,0.6);
         \draw (3,0) circle (0.2cm);
         \draw (5,0) circle (0.2cm);
         \draw (0.4,0.2) -- (0.6,0);
         \draw (0.6,0) -- (0.4,-0.2);
         \node at (0.55,-0.4) {\scriptsize $a_1$};
         \node at (2.45,-0.3) {\scriptsize $a_2$};
         \node at (4.45,-0.3) {\scriptsize $a_3$};
         \node at (6.45,-0.3) {\scriptsize $a_4$};
          \node at (3.55,0.4) {\scriptsize $v_1$};
         \node at (5.45,0.4) {\scriptsize $v_2$};
         \node at (3,-3) {$D_2$};
      \end{tikzpicture}
    \end{subfigure}
     \begin{subfigure}{}
         \begin{tikzpicture}[scale=0.42]
          \draw (1,0.5) -- (1,-0.5);
          \draw[white, line width=6pt] (0.9,0) -- (1.1,0);
           \draw (0,0) --(11,0);
           \draw[white, line width=6pt] (2,0.5) -- (2,-0.5);
          \draw (2,0.5) -- (2,-0.5);
          \draw (3,0.5) -- (3,-0.5);
          \draw[white, line width=6pt] (4,0.5) -- (4,-0.5);
          \draw (4,0.5) -- (4,-0.5);
          \draw (5,0.5) -- (5,-0.5);
          \draw[white, line width=6pt] (6,0.5) -- (6,-0.5);
          \draw (6,0.5) -- (6,-0.5);
          \draw (7,0.5) -- (7,-0.5);
          \draw[white, line width=6pt] (8,0.5) -- (8,-0.5);
          \draw (8,0.5) -- (8,-0.5);
           \draw (9,0.5) -- (9,-0.5);
           \draw[white, line width=6pt] (10,0.5) -- (10,-0.5);
          \draw (10,0.5) -- (10,-0.5);
         \draw (0,0) to[out=-100, in=-100 ] (10,-0.5);
          \draw (1,0.5) to[out=80 , in=100 ] (2,0.5);
        \draw (2,-0.5) to[out=-80 , in=-100 ] (3,-0.5);
         \draw (3,0.5) to[out=80 , in=100 ] (4,0.5);
        \draw (4,-0.5) to[out=-80 , in=-100 ] (5,-0.5);
         \draw (5,0.5) to[out=80 , in=100 ] (6,0.5);
        \draw (6,-0.5) to[out=-80 , in=-100 ] (7,-0.5);
         \draw (7,0.5) to[out=80 , in=100 ] (8,0.5);
          \draw (1,-0.5) to[out=-80 , in=-40 ] (11,0);
        \draw (8,-0.5) to[out=-80 , in=-100 ] (9,-0.5);
         \draw (9,0.5) to[out=80 , in=100 ] (10,0.5);
        
        \draw (8,-2.5) circle (0.4cm);
         \draw (9.5,-0.2) -- (9.5,0.2);
        \draw (1.5,0.2) -- (1.5,-0.2);
         \draw (1.5,1) -- (1.5,0.6);
         \draw (3,0) circle (0.2cm);
         \draw (5,0) circle (0.2cm);
         \draw (7,0) circle (0.2cm);
         \draw (9,0) circle (0.2cm);
         \draw (0.4,0.2) -- (0.6,0);
         \draw (0.6,0) -- (0.4,-0.2);
         \node at (0.55,-0.4) {\scriptsize $a_1$};
         \node at (2.45,-0.3) {\scriptsize $a_2$};
         \node at (4.45,-0.3) {\scriptsize $a_3$};
         \node at (6.45,-0.3) {\scriptsize $a_4$};
         \node at (8.45,-0.3) {\scriptsize $a_5$};
         \node at (10.45,-0.3) {\scriptsize $a_6$};
          \node at (3.55,0.4) {\scriptsize $v_1$};
         \node at (5.45,0.4) {\scriptsize $v_2$};
         \node at (7.45,0.4) {\scriptsize $v_3$};
         \node at (9.45,0.4) {\scriptsize $v_4$};
         \draw[dotted, gray] (6.5,1.5) rectangle (10.5,-1);
         \node at (6,-4) {$D_3$};
      \end{tikzpicture}
    \end{subfigure}
    \begin{subfigure}{}
    \begin{tikzpicture}[scale=0.42]
          \draw (1,0.5) -- (1,-0.5);
          \draw[white, line width=6pt] (0.9,0) -- (1.1,0);
           \draw (0,0) --(11,0);
           \draw[white, line width=6pt] (2,0.5) -- (2,-0.5);
          \draw (2,0.5) -- (2,-0.5);
          \draw (3,0.5) -- (3,-0.5);
          \draw[white, line width=6pt] (4,0.5) -- (4,-0.5);
          \draw (4,0.5) -- (4,-0.5);
          \draw (5,0.5) -- (5,-0.5);
          \draw[white, line width=6pt] (6,0.5) -- (6,-0.5);
          \draw (6,0.5) -- (6,-0.5);
          \draw (7,0.5) -- (7,-0.5);
          \draw[white, line width=6pt] (8,0.5) -- (8,-0.5);
          \draw (8,0.5) -- (8,-0.5);
           \draw (9,0.5) -- (9,-0.5);
           \draw[white, line width=6pt] (10,0.5) -- (10,-0.5);
          \draw (10,0.5) -- (10,-0.5);
          \draw[dotted] (11,0) -- (12,0);
          \draw (12,0) -- (14,0);
         \draw[white, line width=6pt] (14,0.5) -- (14,-0.5);
          \draw (14,0.5) -- (14,-0.5);
          \draw (14.2,0) -- (15,0);
         \draw (0,0) to[out=-100, in=-100 ] (14,-0.5);
          \draw (1,0.5) to[out=80 , in=100 ] (2,0.5);
        \draw (2,-0.5) to[out=-80 , in=-100 ] (3,-0.5);
         \draw (3,0.5) to[out=80 , in=100 ] (4,0.5);
        \draw (4,-0.5) to[out=-80 , in=-100 ] (5,-0.5);
         \draw (5,0.5) to[out=80 , in=100 ] (6,0.5);
        \draw (6,-0.5) to[out=-80 , in=-100 ] (7,-0.5);
         \draw (7,0.5) to[out=80 , in=100 ] (8,0.5);
         \draw (1,-0.5) to[out=-80 , in=-40 ] (15,0);
        \draw (13,-0.5) -- (13,0.5);
         \draw (13,0.5) to[out=80 , in=100 ] (14,0.5);
        \draw (8,-0.5) to[out=-80 , in=-100 ] (9,-0.5);
         \draw (9,0.5) to[out=80 , in=100 ] (10,0.5);
        \draw (10,-0.5) to[out=-80 , in=-100 ] (11,-0.5);
        \draw (12,-0.5) to[out=-80 , in=-120 ] (13,-0.5);
        
        \draw (11.5,-3.4) circle (0.5cm);
         \draw (13.5,-0.2) -- (13.5,0.2);
        \draw (1.5,0.2) -- (1.5,-0.2);
         \draw (1.5,1) -- (1.5,0.6);
         \draw (3,0) circle (0.2cm);
         \draw (5,0) circle (0.2cm);
         \draw (7,0) circle (0.2cm);
         \draw (9,0) circle (0.2cm);
         \draw (13,0) circle (0.2cm);
         \draw (0.4,0.2) -- (0.6,0);
         \draw (0.6,0) -- (0.4,-0.2);
         \node at (0.55,-0.4) {\scriptsize $a_1$};
         \node at (2.45,-0.3) {\scriptsize $a_2$};
         \node at (4.45,-0.3) {\scriptsize $a_3$};
         \node at (6.45,-0.3) {\scriptsize $a_4$};
         \node at (8.45,-0.3) {\scriptsize $a_5$};
         \node at (10.45,-0.3) {\scriptsize $a_6$};
         \node at (14.6,-0.3) {\scriptsize $a_{2n}$};
          \node at (3.55,0.4) {\scriptsize $v_1$};
         \node at (5.45,0.4) {\scriptsize $v_2$};
         \node at (7.45,0.4) {\scriptsize $v_3$};
         \node at (9.45,0.4) {\scriptsize $v_4$};
         \node at (8.45,-0.3) {\scriptsize $a_5$};
         \node at (12.1,0.4) {\scriptsize $v_{2n-2}$};
         \draw[dotted, gray] (6.5,1.5) rectangle (10.5,-1);
         \draw[dotted, gray] (10.5,0.5) -- (11.5,0.5);
         \draw[dotted, gray] (11.5,1.5) rectangle (14.5,-1);
         \draw (11.5,2) -- (14.5,2);
         \draw(14.5,1.8)--(14.5,2);
          \node at (7,-5) {$D_n$};
      \end{tikzpicture}
      \end{subfigure}
    \caption{Twisted Knot $D_n$.}
    \label{fig:placeholder2}
\end{figure}
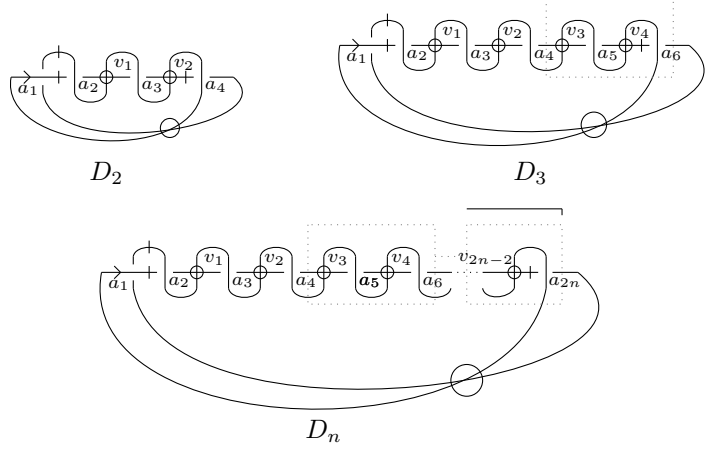

\noindent The diagram $D_n$ contains $2n$ classical crossings. A direct computation of the index values of these crossings yields the values listed in Table 2.

\begin{table}[htpb]
\tbl{Values of $\operatorname{Ind}^1(c)$ and $\operatorname{Ind}^2(c)$ for diagrams presented in Fig.~\ref{fig:placeholder2}. }
{\begin{tabular}{|c| c |c| c | } \toprule
$crossing(c)$ & $\operatorname{Ind}^1(c)$ & $\operatorname{Ind}^2(c)$ & $\operatorname{Ind}^1(c)+\operatorname{Ind}^2(c)$ \\ \hline
$a_1$ & $(2n-3)$ & $(2n-1) $ & $4(n-1)$\\ \hline
$a_2$ & $(2n-1)$ & $-(2n-5)$ & $4$ \\ \hline
$a_3$ & $(2n-3)$ & $-(2n-7)$ & $4$  \\ \hline
$a_4$ & $(2n-5$) & $-(2n-9)$ & $4$  \\ \hline
$a_5$ & $(2n-7)$ & $-(2n-11)$ & $4$  \\ \hline
$\vdots$ & $\vdots$ & $\vdots$ & $\vdots$\\ \hline
$a_{i}$ & $(2n-(2i-3))$ & $-(2n-(2i+1))$  & $4$  \\ \hline
$a_{i+1}$ & $-(2n-(2i+1))$ & $(2n-(2i-3))$ & $4$  \\ \hline
$\vdots$ & $\vdots$ & $\vdots$ & $\vdots$ \\ \hline
$a_{2n-2}$ & $-(2n-7)$ & $(2n-3)$ & $4$  \\ \hline
$a_{2n-1}$ & $-(2n-5)$ & $(2n-1)$ & $4$  \\ \hline
$a_{2n}$ & $(2n-1)$ & $(2n-3)$ & $4(n-1)$ \\ \hline
\end{tabular}
\label{tab:tb}} 
\vspace{0.5cm}
where, $2\leq i \leq n$.
\end{table}

\noindent Since all crossings of $D_n$ are positive, we have
\[
\mathcal{B}_{D_n}(t)
 = \sum_{c\in C(D_n)}
 \left(t^{\operatorname{Ind}^1(c)+\operatorname{Ind}^2(c)}-1\right).
\]

\noindent From Table 2, the crossings $a_1$ and $a_{2n}$ satisfy
\[
\operatorname{Ind}^1(c)+\operatorname{Ind}^2(c)=4(n-1),
\]
while each of the remaining $2n-2$ crossings satisfies
\[
\operatorname{Ind}^1(c)+\operatorname{Ind}^2(c)=4.
\]

\noindent Therefore,
\[
\begin{aligned}
\mathcal{B}_{D_n}(t)
&=2\bigl(t^{4(n-1)}-1\bigr)
 +(2n-2)\bigl(t^4-1\bigr)\\
&=2t^{4(n-1)}+2(n-1)t^4-2n.
\end{aligned}
\]

\noindent This completes the proof.
\end{proof}

\begin{proposition}\label{prop:arcshiftDn}
For the family of twisted knots $\{D_n\}_{n>1}$ shown in Fig.~\ref{fig:placeholder2}, the arc shift number satisfies 
\[ A(D_n)=n. \] 
\end{proposition}
\begin{proof}
To determine the arc shift number of $D_n$, we apply Theorem~\ref{thm:prearc}, which states that for any twisted knot $K$,
\[
\frac{|J(K)|}{2}\leq A(K),
\]
where $J(K)$ denotes the odd writhe of $K$.
Since all $2n$ classical crossings of $D_n$ are positive and odd, it follows that
\[
J(D_n)=2n.
\]
Therefore,
\[
A(D_n)\geq \frac{|J(D_n)|}{2}
       = \frac{2n}{2}
       = n.
\]
Thus, the arc shift number of $D_n$ is bounded below by $n$.
To show that this lower bound is sharp, we construct an explicit sequence of arc shift moves that transforms $D_n$ into the trivial twisted knot with a bar using exactly $n$ arc shift moves.\\

\noindent First, apply an arc shift move to the arc segment labeled $1$ in Fig.~\ref{fig:T}. The resulting diagram, denoted by $D^{1}$, simplifies via one RII move and one VRII move, removing two classical crossings and two virtual crossings.\\

\noindent Next, apply an arc shift move to the arc segment labeled $2$ of $D^{1}$. After simplifying the resulting diagram by the same sequence of moves, we obtain a twisted knot $D^{2}$ with two fewer classical crossings and two fewer virtual crossings than $D^{1}$.\\

\noindent Proceeding inductively, after $n-1$ arc shift moves, we obtain a twisted knot $D^{n-1}$ consisting of two classical crossings, one virtual crossing, and three bars. A final arc shift move transforms $D^{n-1}$ into the trivial twisted knot with a bar.\\

\noindent Hence, $D_n$ can be reduced to the trivial twisted knot with a bar using exactly $n$ arc shift moves. Therefore,
\[
A(D_n)\leq n.
\]

\noindent Combining this inequality with the previously established lower bound
\[
A(D_n)\geq n,
\]
we conclude that
\[
A(D_n)=n.
\]
 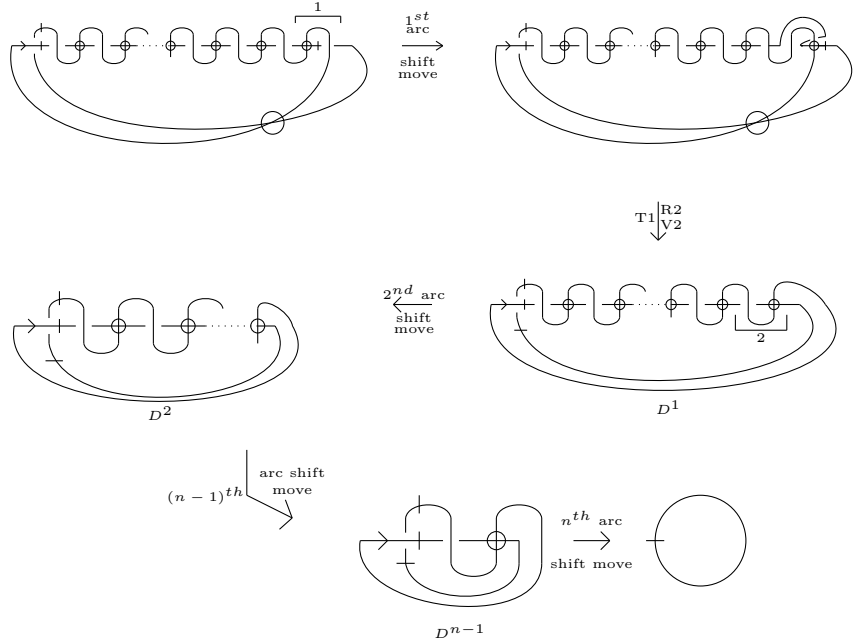
\begin{figure}[htbp]
    \centering
    \begin{subfigure}{}
        \begin{tikzpicture}[scale=0.3]
          \draw (1,0.5) -- (1,-0.5);
          \draw[white, line width=6pt] (0.9,0) -- (1.1,0);
           \draw (0,0) --(5.5,0);
           \draw (7,0) -- (14,0);
           \draw[white, line width=6pt] (2,0.5) -- (2,-0.5);
          \draw (2,0.5) -- (2,-0.5);
          \draw (3,0.5) -- (3,-0.5);
          \draw[white, line width=6pt] (4,0.5) -- (4,-0.5);
          \draw (4,0.5) -- (4,-0.5);
          \draw (5,0.5) -- (5,-0.5);
          \draw (7,0.5) -- (7,-0.5);
           \draw[white, line width=6pt] (8,0.5) -- (8,-0.5);
          \draw (8,0.5) -- (8,-0.5);
          \draw (9,0.5) -- (9,-0.5);
          \draw[white, line width=6pt] (10,0.5) -- (10,-0.5);
          \draw (10,0.5) -- (10,-0.5);

           \draw (11,0.5) -- (11,-0.5);
           \draw[white, line width=6pt] (12,0.5) -- (12,-0.5);
          \draw (12,0.5) -- (12,-0.5);
          \draw[dotted] (5.5,0) -- (7,0);
          
         \draw[white, line width=6pt] (14,0.5) -- (14,-0.5);
          \draw (14,0.5) -- (14,-0.5);
          \draw (14.2,0) -- (15,0);
         \draw (0,0) to[out=-100, in=-100 ] (14,-0.5);
          \draw (1,0.5) to[out=80 , in=100 ] (2,0.5);
        \draw (2,-0.5) to[out=-80 , in=-100 ] (3,-0.5);
         \draw (3,0.5) to[out=80 , in=100 ] (4,0.5);
        \draw (4,-0.5) to[out=-80 , in=-100 ] (5,-0.5);
         \draw (5,0.5) to[out=80 , in=100 ] (6,0.5);
         \draw (8,-0.5) to[out=-80 , in=-100 ] (9,-0.5);
         \draw (7,0.5) to[out=80 , in=100 ] (8,0.5);
         \draw (9,0.5) to[out=80 , in=100 ] (10,0.5);
         \draw (1,-0.5) to[out=-80 , in=-40 ] (15,0);
        \draw (13,-0.5) -- (13,0.5);
         \draw (13,0.5) to[out=80 , in=100 ] (14,0.5);
        \draw (10,-0.5) to[out=-80 , in=-100 ] (11,-0.5);
         \draw (11,0.5) to[out=80 , in=100 ] (12,0.5);
        \draw (12,-0.5) to[out=-80 , in=-120 ] (13,-0.5);
        
        \draw (11.5,-3.4) circle (0.5cm);
        \draw (13.5,0.2) -- (13.5,-0.2);
        \draw (1.3,0.2) -- (1.3,-0.2);
        \draw (1.3,1) -- (1.3,0.6);
         \draw (3,0) circle (0.2cm);
         \draw (5,0) circle (0.2cm);
         \draw (7,0) circle (0.2cm);
         \draw (9,0) circle (0.2cm);
         \draw (11,0) circle (0.2cm);
         \draw (13,0) circle (0.2cm);
         \draw (0.4,0.2) -- (0.6,0);
         \draw (0.6,0) -- (0.4,-0.2);
         \draw (12.5,1) -- (12.5,1.3);
         \draw (14.5,1) -- (14.5,1.3);
         \draw (12.5,1.3) -- (14.5,1.3);
         \node at (13.5,1.7) {\tiny 1};
         \draw (17.2,0) -- (19,0);
         \draw (18.7,0.2) -- (19,0);
         \draw (18.7,-0.2) -- (19,0);
         \node at (17.9,1.2) {\tiny $1^{st}$ };
         \node at (17.9,0.7) {\tiny arc};
          \node at (17.9,-0.7) {\tiny shift};
          \node at (17.9,-1.3) {\tiny move};
          \end{tikzpicture}
    \end{subfigure}
   \begin{subfigure}{}
           \begin{tikzpicture}[scale=0.3]
          \draw (1,0.5) -- (1,-0.5);
          \draw[white, line width=6pt] (0.9,0) -- (1.1,0);
           \draw (0,0) --(5.5,0);
           \draw (7,0) -- (12,0);
           \draw[white, line width=6pt] (2,0.5) -- (2,-0.5);
          \draw (2,0.5) -- (2,-0.5);
          \draw (3,0.5) -- (3,-0.5);
          \draw[white, line width=6pt] (4,0.5) -- (4,-0.5);
          \draw (4,0.5) -- (4,-0.5);
          \draw (5,0.5) -- (5,-0.5);
          \draw (7,0.5) -- (7,-0.5);
           \draw[white, line width=6pt] (8,0.5) -- (8,-0.5);
          \draw (8,0.5) -- (8,-0.5);
          \draw (9,0.5) -- (9,-0.5);
          \draw[white, line width=6pt] (10,0.5) -- (10,-0.5);
          \draw (10,0.5) -- (10,-0.5);

           \draw (11,0.5) -- (11,-0.5);
           \draw[white, line width=6pt] (12,0.5) -- (12,-0.5);
           \draw[dotted] (5.5,0) -- (7,0);
          
          \draw (12,0) -- (12.5,0);
          \draw (12.5,0) to[out=90 , in=180 ] (13.5,1.25);
          \draw (13.5,1.25) to[out=0 , in=90 ] (14.5, 0.4);
         
          \draw (13.5,0) to[out=180, in=180] (14.5,0.4);
        \draw[white, line width=3pt] (14,0.5) -- (14,-0.5);
          \draw (14,0.5) -- (14,-0.5);
           \draw (14.5,0) to[out=180, in=180] (13.5,0);
          \draw (14.2,0) -- (15,0);
          \draw (12,0.5) -- (12,-0.5);
         \draw (0,0) to[out=-100, in=-100 ] (14,-0.5);
          \draw (1,0.5) to[out=80 , in=100 ] (2,0.5);
        \draw (2,-0.5) to[out=-80 , in=-100 ] (3,-0.5);
         \draw (3,0.5) to[out=80 , in=100 ] (4,0.5);
        \draw (4,-0.5) to[out=-80 , in=-100 ] (5,-0.5);
         \draw (5,0.5) to[out=80 , in=100 ] (6,0.5);
         \draw (8,-0.5) to[out=-80 , in=-100 ] (9,-0.5);
         \draw (7,0.5) to[out=80 , in=100 ] (8,0.5);
         \draw (9,0.5) to[out=80 , in=100 ] (10,0.5);
         \draw (1,-0.5) to[out=-80 , in=-40 ] (15,0);
        \draw (13,-0.5) -- (13,0.5);
         \draw (13,0.5) to[out=80 , in=100 ] (14,0.5);
        \draw (10,-0.5) to[out=-80 , in=-100 ] (11,-0.5);
         \draw (11,0.5) to[out=80 , in=100 ] (12,0.5);
        \draw (12,-0.5) to[out=-80 , in=-120 ] (13,-0.5);
        
        \draw (11.5,-3.4) circle (0.5cm);
         \draw (14.5,0.2) -- (14.5,-0.2);
        \draw (1.3,1) -- (1.3,0.6);
        \draw (1.3,0.2) -- (1.3,-0.2);
         \draw (3,0) circle (0.2cm);
         \draw (5,0) circle (0.2cm);
         \draw (7,0) circle (0.2cm);
         \draw (9,0) circle (0.2cm);
         \draw (11,0) circle (0.2cm);
         \draw (14,0) circle (0.2cm);
         \draw (0.4,0.2) -- (0.6,0);
         \draw (0.6,0) -- (0.4,-0.2);
          \end{tikzpicture}
    \end{subfigure}
    \begin{subfigure}{}
          \begin{tikzpicture}[scale=0.46]
          \draw (1,0.5) -- (1,-0.5);
          \draw[white, line width=6pt] (0.9,0) -- (1.1,0);
           \draw (0,0) --(5.5,0);
           \draw (7,0) -- (7.5,0);
           \draw[white, line width=6pt] (2,0.5) -- (2,-0.5);
          \draw (2,0.5) -- (2,-0.5);
          \draw (3,0.5) -- (3,-0.5);
          \draw[white, line width=6pt] (4,0.5) -- (4,-0.5);
          \draw (4,0.5) -- (4,-0.5);
          \draw (5,0.5) -- (5,-0.5);
          \draw (7,0.5) -- (7,-0.5);
           \draw[dotted] (5.5,0) -- (7,0);
         
         \draw (0,0) to[out=-100, in=-60 ] (8,0);
          \draw (1,0.5) to[out=80 , in=100 ] (2,0.5);
        \draw (2,-0.5) to[out=-80 , in=-100 ] (3,-0.5);
         \draw (3,0.5) to[out=80 , in=100 ] (4,0.5);
        \draw (4,-0.5) to[out=-80 , in=-100 ] (5,-0.5);
         \draw (5,0.5) to[out=80 , in=100 ] (6,0.5);
         \draw (7,0.5) to[out=80 , in=100 ] (8,0);
         \draw (1,-0.5) to[out=-80 , in=-60 ] (7.5,0);
        \draw (1.3,1) -- (1.3,0.6);
        \draw (1.4,-1) -- (0.9,-1);
        \draw (1.3,0.2) -- (1.3,-0.2);
         \draw (3,0) circle (0.2cm);
         \draw (5,0) circle (0.2cm);
         \draw (7,0) circle (0.2cm);
         \draw (0.4,0.2) -- (0.6,0);
         \draw (0.6,0) -- (0.4,-0.2); 
         \node at (4.2,-2.5) {\tiny $D^2$};
          \end{tikzpicture}
    \end{subfigure}
      \begin{subfigure}{}
           \begin{tikzpicture}[scale=0.34]
          \draw (1,0.5) -- (1,-0.5);
          \draw[white, line width=6pt] (0.9,0) -- (1.1,0);
           \draw (0,0) --(5.5,0);
           \draw (7,0) -- (12,0);
           \draw[white, line width=6pt] (2,0.5) -- (2,-0.5);
          \draw (2,0.5) -- (2,-0.5);
          \draw (3,0.5) -- (3,-0.5);
          \draw[white, line width=6pt] (4,0.5) -- (4,-0.5);
          \draw (4,0.5) -- (4,-0.5);
          \draw (5,0.5) -- (5,-0.5);
          \draw (7,0.5) -- (7,-0.5);
           \draw[white, line width=6pt] (8,0.5) -- (8,-0.5);
          \draw (8,0.5) -- (8,-0.5);
          \draw (9,0.5) -- (9,-0.5);
          \draw[white, line width=6pt] (10,0.5) -- (10,-0.5);
          \draw (10,0.5) -- (10,-0.5);

           \draw (11,0.5) -- (11,-0.5);
           \draw[dotted] (5.5,0) -- (7,0);
          
         \draw (0,0) to[out=-100, in=-50 ] (13,0);
          \draw (1,0.5) to[out=80 , in=100 ] (2,0.5);
        \draw (2,-0.5) to[out=-80 , in=-100 ] (3,-0.5);
         \draw (3,0.5) to[out=80 , in=100 ] (4,0.5);
        \draw (4,-0.5) to[out=-80 , in=-100 ] (5,-0.5);
         \draw (5,0.5) to[out=80 , in=100 ] (6,0.5);
         \draw (8,-0.5) to[out=-80 , in=-100 ] (9,-0.5);
         \draw (7,0.5) to[out=80 , in=100 ] (8,0.5);
         \draw (9,0.5) to[out=80 , in=100 ] (10,0.5);
         \draw (1,-0.5) to[out=-80 , in=-40 ] (12,0);
        \draw (10,-0.5) to[out=-80 , in=-100 ] (11,-0.5);
         \draw (11,0.5) to[out=80 , in=120 ] (13,0);
        
        \draw (1.3,0.2) -- (1.3,-0.2);
        \draw (1.3,1) -- (1.3,0.6);
        \draw (1.4,-1) -- (0.9,-1);
         \draw (3,0) circle (0.2cm);
         \draw (5,0) circle (0.2cm);
         \draw (7,0) circle (0.2cm);
         \draw (9,0) circle (0.2cm);
         \draw (11,0) circle (0.2cm);
         \draw (0.4,0.2) -- (0.6,0);
         \draw (0.6,0) -- (0.4,-0.2);
         \draw (6.5,4) -- (6.5,2.5);
         \draw (6.2,2.8) -- (6.5,2.5);
         \draw (6.8,2.8) -- (6.5,2.5);
         \node at (7,3.7) {\tiny R2};
         \node at (7,3.1) {\tiny V2};
         \node at (6,3.4) {\tiny T1};
         \node at (6.9,-4) {\tiny $D^1$};
         \draw (-3.8,0) -- (-2.3,0);
         \draw (-3.8,0) -- (-3.5,0.3);
         \draw (-3.8,0) -- (-3.5,-0.3);
          \node at (-3,0.5) {\tiny $2^{nd}$ arc};
          \node at (-3,-0.5) {\tiny shift};
          \node at (-3,-1) {\tiny move};
          \draw (11.5,-0.4) -- (11.5,-1);
           \draw (9.5,-0.4) -- (9.5,-1);
           \draw (9.5,-1) -- (11.5,-1);
           \node at (10.5,-1.3) {\tiny 2};
          \end{tikzpicture}
      \end{subfigure}
      \begin{subfigure}{}
      \centering
          \begin{tikzpicture}[scale=0.6]
          \draw (1,0.5) -- (1,-0.5);
          \draw[white, line width=6pt] (0.9,0) -- (1.1,0);
           \draw (0,0) --(3.5,0);
           \draw[white, line width=6pt] (2,0.5) -- (2,-0.5);
          \draw (2,0.5) -- (2,-0.5);
          \draw (3,0.5) -- (3,-0.5);
          \draw[white, line width=6pt] (4,0.5) -- (4,-0.5);
          \draw (3.5,0) -- (3.5,-0.5);
         \draw (0,0) to[out=-100, in=-90 ] (4,-0.5);
         \draw (4,0.5) -- (4,-0.5);
          \draw (1,0.5) to[out=80 , in=100 ] (2,0.5);
        \draw (2,-0.5) to[out=-80 , in=-100 ] (3,-0.5);
         \draw (3,0.5) to[out=80 , in=100 ] (4,0.5);
         \draw (1,-0.5) to[out=-80 , in=-90 ] (3.5,-0.5);
        \draw (1.3,0.2) -- (1.3,-0.2);
         \draw (1.3,1) -- (1.3,0.6);
        \draw (1.2,-0.5) -- (0.8,-0.5);
         \draw (3,0) circle (0.2cm);
         \draw (0.4,0.2) -- (0.6,0);
         \draw (0.6,0) -- (0.4,-0.2); 
         \node at (2.2,-2) {\tiny $D^{n-1}$};
         \draw (-2.5,2) --(-2.5,1);
         \draw (-2.5,1) --(-1.5,0.5);
         \draw (-1.65,0.9) --(-1.5,0.5);
         \draw (-1.78,0.25) --(-1.5,0.5);
         \node at (-3.4, 1) {\tiny $(n-1)^{th}$ };
          \node at (-1.5, 1.5) {\tiny arc shift };
          \node at (-1.5, 1.1) {\tiny move };
          \node at (5.1,0.5) {\tiny $n^{th}$ arc};
           \node at (5.1,-0.5) {\tiny shift move};
          \draw(4.7,0)--(5.5,0);
          \draw(5.3,0.2)--(5.5,0);
          \draw(5.3,-0.2)--(5.5,0);
          \draw (7.5,0) circle (1cm);
          \draw (6.3,0)--(6.7,0);
          \end{tikzpicture}
           \label{fig:T*}
      \end{subfigure}
       \caption{Arc shift moves on $D_n$.}
       \label{fig:T}
       \end{figure}

\end{proof}

\begin{remark}
For the family of twisted knots $\{D_n\}_{n>1}$ defined in Proposition~\ref{prop:Dn}, the arc shift number is given by
\[
A(D_n)=n=\frac{1}{2}\left|\text{constant term of }\mathcal{B}_{D_n}(t)\right|.
\]
Thus, for this family, the arc shift number can be recovered directly from the $\mathcal{B}$-Index polynomial.
\end{remark}

\noindent Now we observe the behavior of the $\mathcal{B}$-Index polynomial under reflection and orientation reversing. 
 \begin{theorem} \label{th-mirror}
Let $D$ be an oriented twisted knot diagram, and let $D^*$ denote its mirror image. Then, we have 
$$
\mathcal{B}_{D^*}(t) = -\mathcal{B}_{D}(t). 
$$
\end{theorem}
\begin{proof} Let $c$ be a classical crossing of $D$, and let $c^*$  denote the  corresponding crossing in $D^*$. Let $D_{c}$ and $D_{c^*}^{*}$ be the diagrams obtained from $D$ and $D^*$, respectively, by smoothing the crossings $c$ and $c^*$. Then the base point on the first component of $D_{c^*}^{*}$ (respectively, second component) corresponds to the base point on the second component (respectively, on the first component) of $D_{c}$, as illustrated in Fig.~\ref{fig:mirror}. 
Note that $D_{c}$ and $D_{c^*}^{*}$ are mirror images of each other and differ only in the placement of the base points. The arc $C$ containing the base point $\alpha$ in $D_c$ corresponds to the arc containing the base point $\beta$ in $D_{c^*}^{*}$, as shown in Fig.~\ref{fig:mirror}.

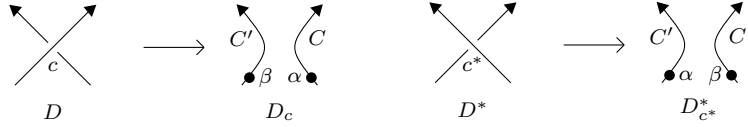
\begin{figure}[htbp]
    \centering
    \begin{subfigure}
    \centering
    \begin{tikzpicture}
    \draw (-0.5,0.5) -- (0.5,-0.5);
    \draw[white, line width=5pt] (-0.5,-0.5) -- (0.5,0.5);
     \draw (-0.5,-0.5) -- (0.5,0.5);
       \node at (0.5,0.5) [
  fill=black,
  regular polygon,
  regular polygon sides=3,
  rotate=70,
   scale=0.3
] {};
\node at (-0.5,0.5) [
  fill=black,
  regular polygon,
  regular polygon sides=3,
  rotate=50,
   scale=0.3
] {};
     \draw (1.2,0) -- (2,0);
     \draw (1.9,0.1) -- (2,0);
     \draw (1.9,-0.1) -- (2,0);
     \draw (2.5,0.5) .. controls (2.9,0) .. (2.5,-0.5);
     \draw (3.5,0.5) .. controls (3.1,0) .. (3.5,-0.5);
      \node at (2.5,0.5) [
  fill=black,
  regular polygon,
  regular polygon sides=3,
  rotate=45,
   scale=0.3
] {};
      \node at (3.5,0.5) [
  fill=black,
  regular polygon,
  regular polygon sides=3,
  rotate=-45,
   scale=0.3
] {};
      \fill (2.6,-0.4) circle (2pt) node[ right ] {\footnotesize $\beta$};
      \fill (3.42,-0.4) circle (2pt) node[left ] {\footnotesize $\alpha$};
      \node at (0,-0.25) {\footnotesize $c$};
      \node at (0,-0.85) {\footnotesize $D$};
      \node at (3,-0.85) {\footnotesize $D_{c}$};
      \node at (2.5,0.125) {\footnotesize $C'$};
       \node at (3.5,0.125) {\footnotesize $C$};
       
\end{tikzpicture}
\end{subfigure}\hspace{1cm}
\begin{subfigure}
    \centering
    \begin{tikzpicture}
     \draw (-0.5,-0.5) -- (0.5,0.5);
     \draw[white, line width=5pt] (-0.5,0.5) -- (0.5,-0.5);
     \draw (-0.5,0.5) -- (0.5,-0.5);
       \node at (0.5,0.5) [
  fill=black,
  regular polygon,
  regular polygon sides=3,
  rotate=70,
   scale=0.3
] {};
\node at (-0.5,0.5) [
  fill=black,
  regular polygon,
  regular polygon sides=3,
  rotate=50,
   scale=0.3
] {};
     \draw (1.2,0) -- (2,0);
     \draw (1.9,0.1) -- (2,0);
     \draw (1.9,-0.1) -- (2,0);
     \draw (2.5,0.5) .. controls (2.9,0) .. (2.5,-0.5);
     \draw (3.5,0.5) .. controls (3.1,0) .. (3.5,-0.5);
      \node at (2.5,0.5) [
  fill=black,
  regular polygon,
  regular polygon sides=3,
  rotate=45,
   scale=0.3
] {};
      \node at (3.5,0.5) [
  fill=black,
  regular polygon,
  regular polygon sides=3,
  rotate=-45,
   scale=0.3
] {};
      \fill (2.6,-0.4) circle (2pt) node[ right ] {\footnotesize $\alpha$};
      \fill (3.42,-0.4) circle (2pt) node[left ] {\footnotesize $\beta$};
      \node at (0,-0.25) {\footnotesize $c^*$};
      \node at (0,-0.85) {\footnotesize $D^*$};
      \node at (3,-0.85) {\footnotesize $D^{*}_{c^*}$};
      \node at (2.5,0.125) {\footnotesize $C'$};
       \node at (3.5,0.125) {\footnotesize $C$};
       
\end{tikzpicture}
\end{subfigure}
    \caption{Diagrams $D$, $D_c$, $D^*$, and $D^{*}_{c^*}$.}
    \label{fig:mirror}
\end{figure}

\noindent If $c_i^{*}$ be the crossing in $D_{c^*}^{*}$ corresponding to the crossing $c_i$ of $D_{c}$. Then, 
 \[\operatorname{sgn}(c_{i}^{*})=-\operatorname{sgn}(c_{i}), \quad \text{ and}\]
\[\mathcal{U}_{L}(D^*_{c^*}:C)=\{c_{i}^{*}\in D^{*}_{c^*}~ | \quad c_{i}\in \mathcal{O}_{L}(D_c:C)\}, \quad \text{and}\] 
\[\mathcal{O}_{L}(D^*_{c^*}:C)=\{c_{i}^{*}\in D^{*}_{c^*}~ | \quad c_{i}\in \mathcal{U}_{L}(D_c:C)\}.\] 
Therefore, we have 
 \begin{align*}
\operatorname{Ind}^1(c^*)&=\displaystyle \sum_{c'\in \mathcal{O}_{L}(D^{*}_{c^*}:C) }\operatorname{sgn}(c')- \displaystyle \sum_{c'\in \mathcal{U}_{L}(D^{*}_{c^*}:C) }\operatorname{sgn}(c'),\\
&= \displaystyle \sum_{c\in \mathcal{U}_{L}(D_{c}:C) }(-\operatorname{sgn}(c))- \displaystyle \sum_{c\in \mathcal{O}_{L}(D_{c}:C) }(-\operatorname{sgn}(c)),\\
&=\operatorname{Ind}^2(c).
\end{align*}
Similarly, we can show that $\operatorname{Ind}^2(c^*)=\operatorname{Ind}^1(c)$. Hence, the $\mathcal{B}$-Index polynomial of $D^*$ is given by
\begin{align*}
\mathcal{B}_{D^*}(t)&=\displaystyle \sum_{c^*\in \mathcal{C}(D^{*})}\operatorname{sgn}(c^*)(t^{\operatorname{Ind}^1(c^*)+\operatorname{Ind}^2(c^*)}-1),\\
&=  \displaystyle \sum_{c\in \mathcal{C}(D)}(-\operatorname{sgn}(c))(t^{\operatorname{Ind}^2(c)+\operatorname{Ind}^1(c)}-1)= - \mathcal{B}_{D}(t).
\end{align*}
Hence, proved the desired result.
\end{proof}

 \begin{theorem} \label{th-rev_orient}
Let $D$ be an oriented twisted knot diagram. If, for every classical crossing $c$ of $D$, each component of the smoothed diagram $D_c$ contains either exactly one arc or an even number of arcs, then
\[
\mathcal{B}_{\overline{D}}(t)=\mathcal{B}_{D}(t),
\]
where $\overline{D}$ denotes the diagram obtained from $D$ by reversing its orientation. 
\end{theorem}
\begin{proof}
Let $c$ be a classical crossing of $D$, and let $D_c$ be the diagram obtained from $D$ by smoothing $c$ along the orientation. By hypothesis, each component of $D_c$ contains either exactly one arc or an even number of arcs. Let $\overline{D}$ be the diagram obtained from $D$ by reversing its orientation, and let $\overline{c}$ denote the crossing of $\overline{D}$ corresponding to $c$. Further, let $\overline{D}_{\overline{c}}$ be the diagram obtained from $\overline{D}$ by smoothing the crossing $\overline{c}$ along the orientation.\\

\noindent Note that $D_c$ and $\overline{D}_{\overline{c}}$ are reverses of each other, differing only in the placement of their base points. More precisely, the base points on the first and second components are interchanged between $D_c$ and $\overline{D}_{\overline{c}}$, as shown in Fig.~\ref{fig:inv}.

\begin{figure}[htbp]
    \centering
    \begin{subfigure}
    \centering
    \begin{tikzpicture}
    \draw (-0.5,0.5) -- (0.5,-0.5);
    \draw[white, line width=5pt] (-0.5,-0.5) -- (0.5,0.5);
     \draw (-0.5,-0.5) -- (0.5,0.5);
       \node at (0.5,0.5) [
  fill=black,
  regular polygon,
  regular polygon sides=3,
  rotate=70,
   scale=0.3
] {};
\node at (-0.5,0.5) [
  fill=black,
  regular polygon,
  regular polygon sides=3,
  rotate=50,
   scale=0.3
] {};
     \draw (1.2,0) -- (2,0);
     \draw (1.9,0.1) -- (2,0);
     \draw (1.9,-0.1) -- (2,0);
     \draw (2.5,0.5) .. controls (2.9,0) .. (2.5,-0.5);
     \draw (3.5,0.5) .. controls (3.1,0) .. (3.5,-0.5);
      \node at (2.5,0.5) [
  fill=black,
  regular polygon,
  regular polygon sides=3,
  rotate=45,
   scale=0.3
] {};
      \node at (3.5,0.5) [
  fill=black,
  regular polygon,
  regular polygon sides=3,
  rotate=-45,
   scale=0.3
] {};
      \fill (2.6,-0.4) circle (2pt) node[ right ] {\footnotesize $\beta$};
      \fill (3.42,-0.4) circle (2pt) node[left ] {\footnotesize $\alpha$};
      \node at (0,-0.25) {\footnotesize $c$};
      \node at (0,-0.85) {\footnotesize $D$};
      \node at (3,-0.85) {\footnotesize $D_c$};
      \node at (2.5,0.125) {\footnotesize $C'$};
       \node at (3.5,0.125) {\footnotesize $C$};
       
\end{tikzpicture}
\end{subfigure}\hspace{1cm}
\begin{subfigure}
\centering
    \begin{tikzpicture}
    \draw (-0.5,0.5) -- (0.5,-0.5);
    \draw[white, line width=5pt] (-0.5,-0.5) -- (0.5,0.5);
     \draw (-0.5,-0.5) -- (0.5,0.5);
       \node at (-0.5,-0.5) [
  fill=black,
  regular polygon,
  regular polygon sides=3,
  rotate=10,
   scale=0.3
] {};
\node at (0.5,-0.5) [
  fill=black,
  regular polygon,
  regular polygon sides=3,
  rotate=-10,
   scale=0.3
] {};
     \draw (1.2,0) -- (2,0);
     \draw (1.9,0.1) -- (2,0);
     \draw (1.9,-0.1) -- (2,0);
     \draw (2.5,0.5) .. controls (2.9,0) .. (2.5,-0.5);
     \draw (3.5,0.5) .. controls (3.1,0) .. (3.5,-0.5);
      \node at (2.5,-0.5) [
  fill=black,
  regular polygon,
  regular polygon sides=3,
  rotate=30,
   scale=0.3
] {};
      \node at (3.5,-0.5) [
  fill=black,
  regular polygon,
  regular polygon sides=3,
  rotate=-30,
   scale=0.3
] {};
      \fill (2.6,0.4) circle (2pt) node[ right ] {\footnotesize $\alpha$};
      \fill (3.42,0.4) circle (2pt) node[left ] {\footnotesize $\beta$};
      \node at (0,-0.25) {\footnotesize $\bar{c}$};
      \node at (0,-0.85) {\footnotesize $\overline{D}$};
      \node at (3,-0.85) {\footnotesize $\overline{D}_{\overline{c}}$};
      \node at (2.5,-0.125) {\footnotesize $C'$};
       \node at (3.5,-0.125) {\footnotesize $C$};
       
\end{tikzpicture}
\end{subfigure}
    \caption{Diagrams $D$, $D_c$, $\overline{D}$, and $\overline{D}_{\overline{c}}$.}
    \label{fig:inv}
\end{figure}
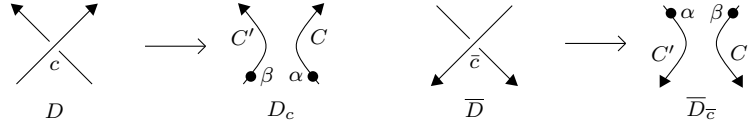

\noindent Let $C$ be the arc shown in Fig.~\ref{fig:inv}, and let
\[
\langle C=C_1, C_2, \ldots, C_n \rangle
\]
denote the sequence of consecutive arcs in the first component of $D_c$ starting from $C$. Since this component contains either exactly one arc or an even number of arcs, it follows that $n=1$ or $n$ is even.
Because the orientation of $\overline{D}_{\overline{c}}$ is the reverse of that of $D_c$, the corresponding sequence of consecutive arcs in the first component of $\overline{D}_{\overline{c}}$ starting from $C$ is
\[
\langle C=C_1, C_n, C_{n-1}, \ldots, C_2 \rangle.
\]

\noindent Let $\overline{c}_i$ denote the crossing in $\overline{D}_{\overline{c}}$ corresponding to the crossing $c_i$ in $D_c$. Then
\[
\operatorname{sgn}(\overline{c}_i)=\operatorname{sgn}(c_i),
\]
and
\[\mathcal{U}_{L}(\overline{D}_{\overline{c}}:C)=
\{\overline{c}_i\in \overline{D}_{\overline{c}}~ | \quad c_{i}\in \mathcal{U}_{L}(D_c:C)\}, \quad \text{and}\] 
\[\mathcal{O}_{L}(\overline{D}_{\overline{c}}:C)=\{\overline{c}_i\in \overline{D}_{\overline{c}}~ | \quad c_{i}\in \mathcal{O}_{L}(D_c:C)\}.\] 
Since the arc $C$, depicted in Fig.~\ref{fig:inv}, contains the base point $\alpha$ in $D_c$ and the base point $\beta$ in $\overline{D}_{\overline{c}}$, it follows that 
 \begin{align*}
\operatorname{Ind}^1(\overline{c})&=\displaystyle \sum_{c'\in \mathcal{O}_{L}(\overline{D}_{\overline{c}}:C) }\operatorname{sgn}(c')- \displaystyle \sum_{c'\in \mathcal{U}_{L}(\overline{D}_{\overline{c}}:C) }\operatorname{sgn}(c')\\
&= \displaystyle \sum_{c\in \mathcal{O}_{L}(D_{c}:C) }\operatorname{sgn}(c)- \displaystyle \sum_{c\in \mathcal{U}_{L}(D_{c}:C) }\operatorname{sgn}(c)\\
&=\operatorname{Ind}^2(c).
\end{align*}
Similarly, we can show that $\operatorname{Ind}^2(\overline{c})=\operatorname{Ind}^1(c)$. Hence, the $\mathcal{B}$-Index polynomial of $\overline{D}$ is given by
\begin{align*}
\mathcal{B}_{\overline{D}}(t)&=\displaystyle \sum_{\overline{c}\in \mathcal{C}(\overline{D})}\operatorname{sgn}(\overline{c})(t^{\operatorname{Ind}^1(\overline{c})+\operatorname{Ind}^2(\overline{c})}-1)\\
&=  \displaystyle \sum_{c\in \mathcal{C}(D)}\operatorname{sgn}(c)(t^{\operatorname{Ind}^2(c)+\operatorname{Ind}^1(c)}-1)= \mathcal{B}_{D}(t).
\end{align*}
Hence, proved the desired result.
  \end{proof}


\noindent If an oriented twisted knot diagram fails to satisfy the conditions of Theorem~\ref{th-rev_orient}, the $B$-index polynomial of the diagram and its orientation reversal are not necessarily identical. We illustrate this non-equivalence, where $\mathcal{B}_{D}(t)\neq \mathcal{B}_{\overline{D}}(t)$, in the following example.

  \begin{figure}[htbp]
       \centering
       \begin{subfigure}{}
           \begin{tikzpicture}[scale=0.6]
           \draw (1,1) .. controls (2.3,3) and (3.3,3) .. (1.5,0); 
                \draw[white , line width=4pt] (2.3,1.5)-- (2.4,1.8);
           \draw (0,1) .. controls (-1,2) and (3,3) .. (3,0)
      .. controls (3,-2) and (-1,-2) .. (0,0);
           \draw (0,0)--(1,1);
               \draw[white, line width=5pt] (0,1) -- (1,0);
                \draw (0,1) -- (1,0);
                \draw (1,0) -- (2,-1);
                \draw[white, line width=5pt] (1,-0.5) -- (1.5,0);
                \draw (1,-0.5) -- (1.5,0);
                \draw[white, line width=4pt] (1.6,1.97) -- (1.8,1.95);
                \draw (1.87,2.1) -- (1.63,1.85);
                \draw (1,-0.5) .. controls (-0.5,-2.5) and (3,-2.5) .. (2,-1);
                \draw[white, line width=5pt] (0.5,-1.4) -- (0.8,-1.5);
                \draw (0.5,-1.4) -- (0.85,-1.49);
                \draw (2.15,-1.25) circle (4pt); 
                \draw (2.8,0)--(3.2,0);
                \draw (-0.25,-0.2) -- (0.15,-0.2);
                \draw (1.95,1.6) -- (2.1,1.9);
                 \node at (1.5,-1.455) [
  fill=black,
  regular polygon,
  regular polygon sides=3,
  rotate=47,
   scale=0.3
] {};
\node at (1.5,-3) {$(a)D$};
           \end{tikzpicture}
       \end{subfigure}
       \begin{subfigure}{}
           \begin{tikzpicture}[scale=0.6]
           \draw (1,1) .. controls (2.3,3) and (3.3,3) .. (1.5,0); 
                \draw[white , line width=4pt] (2.3,1.5)-- (2.4,1.8);
           \draw (0,1) .. controls (-1,2) and (3,3) .. (3,0)
      .. controls (3,-2) and (-1,-2) .. (0,0);
           \draw (0,0)--(1,1);
               \draw[white, line width=5pt] (0,1) -- (1,0);
                \draw (0,1) -- (1,0);
                \draw (1,0) -- (2,-1);
                \draw[white, line width=5pt] (1,-0.5) -- (1.5,0);
                \draw (1,-0.5) -- (1.5,0);
                \draw[white, line width=4pt] (1.6,1.97) -- (1.8,1.95);
                \draw (1.87,2.1) -- (1.63,1.85);
                \draw (1,-0.5) .. controls (-0.5,-2.5) and (3,-2.5) .. (2,-1);
                \draw[white, line width=5pt] (0.5,-1.4) -- (0.8,-1.5);
                \draw (0.5,-1.4) -- (0.85,-1.49);
                \draw (2.15,-1.25) circle (4pt); 
                \draw (2.8,0)--(3.2,0);
                \draw (-0.25,-0.2) -- (0.15,-0.2);
                \draw (1.95,1.6) -- (2.1,1.9);
                 \node at (1.5,-1.455) [
  fill=black,
  regular polygon,
  regular polygon sides=3,
  rotate=85,
   scale=0.3
] {};
\node at (1.5,-3) {$(b)\overline{D}$};
           \end{tikzpicture}
       \end{subfigure}
         \caption{Diagrams $D$ and $\overline{D}$.}
     \label{fig:expmirror}
   \end{figure}
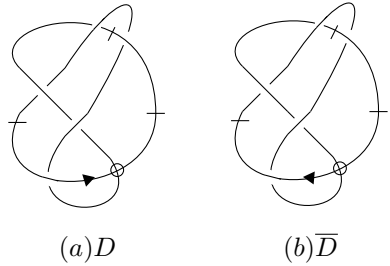
\vspace{-0.5cm}
\begin{example} Consider a twisted knot diagram $D$ with three bars, as shown in Fig.~\ref{fig:expmirror}$(a)$. Let $\overline{D}$ denote the reverse of $D$, and is illustrated in Fig.~\ref{fig:expmirror}$(b)$. The $\mathcal{B}$-Index polynomials of $D$, and $\overline{D}$ are given by
   \begin{align*}
      \mathcal{B}_{D}(t)=&(t^2-1)-4(t^{-1}-1),\\
      \mathcal{B}_{\overline{D}}(t)=&(t^{2}-1)-(t^{-2}-1)-2(t^{2}-1).
   \end{align*}
   Therefore, it follows that 
   \[\mathcal{B}_{\overline{D}}(t)\neq \mathcal{B}_{D}(t).\]
   \end{example}

\section{Cosmetic crossing change conjecture}
\label{sec:cs}
We define a crossing in a twisted knot diagram as \emph{nugatory} if it can be eliminated by twisting a portion of the diagram, see Fig.~\ref{fig20new}.  A classic example of a nugatory crossing is one that can be undone via an RI-move. Consequently, performing a crossing change operation at a nugatory crossing of $D$ yields a diagram that remains equivalent to the original diagram $D$. Fig.~\ref{fig20new} shows a general form for a nugatory crossing. 
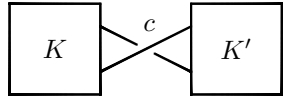
\begin{figure}[!ht]
\centering 
\unitlength=0.6mm
\begin{picture}(0,20)(0,10)
\thicklines
\qbezier(-30,10)(-30,10)(-30,30)
\qbezier(-30,10)(-30,10)(-10,10) 
\qbezier(-10,10)(-10,10)(-10,30)
\qbezier(-30,30)(-30,30)(-10,30)
\put(-20,20){\makebox(0,0)[cc]{$K$}}
\qbezier(30,10)(30,10)(30,30)
\qbezier(30,10)(30,10)(10,10) 
\qbezier(10,10)(10,10)(10,30)
\qbezier(30,30)(30,30)(10,30)
\put(20,20){\makebox(0,0)[cc]{$K'$}}
\qbezier(-10,15)(-10,15)(10,25)
\qbezier(-10,25)(-10,25)(-2,21)
\qbezier(10,15)(10,15)(2,19)
\put(-0.5,24){$c$}
\end{picture}
\caption{Nugatory crossing $c$.} \label{fig20new}
\end{figure}

\begin{definition}
A crossing change in a twisted knot diagram $D$ is said to be \emph{cosmetic} if the new diagram, say $D'$, is isotopic  to $D$.  
\end{definition}

\noindent A crossing change on a nugatory crossing is called a \emph{trivial cosmetic crossing change}.  The following question is still open.

\begin{question} \cite[Problem~1.58]{kirby2problems} 
Do non-trivial cosmetic crossing change exist?
\end{question}

\noindent For virtual knots, this question has been answered in the negative for a wide class of knots.  In \cite[p.~15]{folwaczny2013linking}, L.~Folwaczny and L.~Kauffman proved that a crossing  $c$ in $D$ with $\operatorname{Ind}(c) \neq 0$ is not cosmetic. We explore this question for twisted knots and establish a condition in Theorem~\ref{thm:cosmetic} for a crossing not to be cosmetic.
   
\begin{lemma} \label{lem:cosmetic}
Let $D'$  be the diagram obtained from twisted knot diagram $D$ by switching the crossing $c$, and let $c'$ denote the corresponding crossing in  $D'$. Then 
\[\operatorname{Ind}^1 (c') =  \operatorname{Ind}^2 (c), \quad \text{and} \quad \operatorname{Ind}^2 (c') =  \operatorname{Ind}^1 (c).\]
\end{lemma}

\begin{proof}
Let $D'$ be the diagram obtained from $D$ by applying the crossing change operation at $c$, and let $c'$ denote the corresponding crossing in $D'$, as shown in Fig.~\ref{fig:cro change}.  Then the diagrams $D_c$ and $D'_{c'}$ represent the same diagram, differing only by a reversal of the ordering of their components, see Fig.~\ref{fig:cro change}.
\begin{figure}[htbp]
    \centering
    \begin{subfigure}
    \centering
    \begin{tikzpicture}
    \draw (-0.5,0.5) -- (0.5,-0.5);
    \draw[white, line width=5pt] (-0.5,-0.5) -- (0.5,0.5);
     \draw (-0.5,-0.5) -- (0.5,0.5);
       \node at (0.5,0.5) [
  fill=black,
  regular polygon,
  regular polygon sides=3,
  rotate=70,
   scale=0.3
] {};
\node at (-0.5,0.5) [
  fill=black,
  regular polygon,
  regular polygon sides=3,
  rotate=50,
   scale=0.3
] {};
     \draw (1.2,0) -- (2,0);
     \draw (1.9,0.1) -- (2,0);
     \draw (1.9,-0.1) -- (2,0);
     \draw (2.5,0.5) .. controls (2.9,0) .. (2.5,-0.5);
     \draw (3.5,0.5) .. controls (3.1,0) .. (3.5,-0.5);
      \node at (2.5,0.5) [
  fill=black,
  regular polygon,
  regular polygon sides=3,
  rotate=45,
   scale=0.3
] {};
      \node at (3.5,0.5) [
  fill=black,
  regular polygon,
  regular polygon sides=3,
  rotate=-45,
   scale=0.3
] {};
      \fill (2.6,-0.4) circle (2pt) node[ right ] {\footnotesize $\beta$};
      \fill (3.42,-0.4) circle (2pt) node[left ] {\footnotesize $\alpha$};
      \node at (0,-0.25) {\footnotesize $c$};
      \node at (0,-0.85) {\footnotesize $D$};
      \node at (3,-0.85) {\footnotesize $D_{c}$};
      \node at (2.5,0.125) {\footnotesize $C'$};
       \node at (3.5,0.125) {\footnotesize $C$};
       
\end{tikzpicture}
\end{subfigure}\hspace{1cm}
\begin{subfigure}
    \centering
    \begin{tikzpicture}
     \draw (-0.5,-0.5) -- (0.5,0.5);
     \draw[white, line width=5pt] (-0.5,0.5) -- (0.5,-0.5);
     \draw (-0.5,0.5) -- (0.5,-0.5);
       \node at (0.5,0.5) [
  fill=black,
  regular polygon,
  regular polygon sides=3,
  rotate=70,
   scale=0.3
] {};
\node at (-0.5,0.5) [
  fill=black,
  regular polygon,
  regular polygon sides=3,
  rotate=50,
   scale=0.3
] {};
     \draw (1.2,0) -- (2,0);
     \draw (1.9,0.1) -- (2,0);
     \draw (1.9,-0.1) -- (2,0);
     \draw (2.5,0.5) .. controls (2.9,0) .. (2.5,-0.5);
     \draw (3.5,0.5) .. controls (3.1,0) .. (3.5,-0.5);
      \node at (2.5,0.5) [
  fill=black,
  regular polygon,
  regular polygon sides=3,
  rotate=45,
   scale=0.3
] {};
      \node at (3.5,0.5) [
  fill=black,
  regular polygon,
  regular polygon sides=3,
  rotate=-45,
   scale=0.3
] {};
      \fill (2.6,-0.4) circle (2pt) node[ right ] {\footnotesize $\alpha$};
      \fill (3.42,-0.4) circle (2pt) node[left ] {\footnotesize $\beta$};
      \node at (0,-0.25) {\footnotesize $c'$};
      \node at (0,-0.85) {\footnotesize $D'$};
      \node at (3,-0.85) {\footnotesize $D^{'}_{c'}$};
      \node at (2.5,0.125) {\footnotesize $C'$};
       \node at (3.5,0.125) {\footnotesize $C$};
       
\end{tikzpicture}
\end{subfigure}
    \caption{Diagrams $D$, $D_c$, $D'$, and $D^{'}_{c'}$.}
    \label{fig:cro change}
\end{figure}
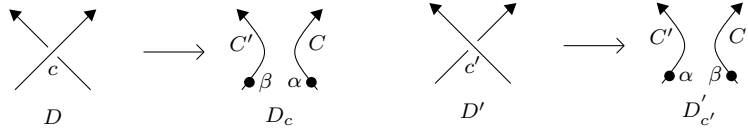
 If $C$ and $C'$ are the two arcs in both $D_c$ and $D'_{c'}$, as illustrated in Fig.~\ref{fig:cro change}.  Then,
\[\mathcal{O}_{L}(D_c:C)=\mathcal{O}_{L}(D'_{c'}:C),\quad \mathcal{U}_{L}(D_{c}:C)=\mathcal{U}_{L}(D'_{c'}:C),\]
\[ \mathcal{O}_{L}(D_{c}:C')=\mathcal{O}_{L}(D'_{c'}:C') \quad \text{and} \quad  \mathcal{U}_{L}(D_c:C')=\mathcal{U}_{L}(D'_{c'}:C').\]
Since in $D_c$ and $D'_{c'}$ the  base point $\alpha$ lies on the first component and is contained in the arcs  $C$ and $C'$, respectively, it follows that
\begin{align*}
\operatorname{Ind}^1(c')&=\displaystyle \sum_{c\in \mathcal{O}_{L}(D_{c'}:C') }\operatorname{sgn}(c)- \displaystyle \sum_{c\in \mathcal{U}_{L}(D_{c'}:C') }\operatorname{sgn}(c)\\
&= \displaystyle \sum_{c\in \mathcal{O}_{L}(D_{c}:C') }\operatorname{sgn}(c)- \displaystyle \sum_{c\in \mathcal{U}_{L}(D_{c}:C') }\operatorname{sgn}(c)=\operatorname{Ind}^2(c).
\end{align*}
By a similar argument, we can show that
\[\operatorname{Ind}^2(c')=\operatorname{Ind}^1(c). \]
Hence, the desired result follows.

\end{proof}

\begin{lemma} \label{lem:cosmetic2}
Let $D'$  be the diagram obtained from twisted knot diagram $D$ by switching the crossing $c$, and let $c'$ denote the corresponding crossing in  $D'$.  For any crossing $c_k$ of $D$ with $c_k\neq c$, let $c'_k$ be the corresponding crossing in $D'$. Then 
\[\operatorname{Ind}^i(c'_k) = \operatorname{Ind}^i (c_k), \quad \text{for~} i=1,2.
\]
\end{lemma}

\begin{proof} 
Let $D'$ be the diagram obtained from $D$ by applying a crossing change operation at $c$, and let $c'$ denote the corresponding crossing in $D'$.  Let $c'_k$ be a crossing in $D'$ corresponding to a crossing $c_k$ in $D$. If  $c_k\neq c$,  then the diagrams $D_{c_{k}}$ and $D'_{c'_{k}}$ differ by exactly one crossing change, while the base points remain on the same components. Therefore, by Lemma~\ref{prop: RIII}, we have
\[\operatorname{Ind}^1(c'_k) = \operatorname{Ind}^1 (c_k), \quad \text{and}\quad \operatorname{Ind}^2(c'_k) = \operatorname{Ind}^2 (c_k). \]
\end{proof}

\begin{theorem} \label{thm:cosmetic}
Let $D$ be a twisted knot diagram and let $c$ be a crossing in $D$. If $\operatorname{Ind}^1(c)+\operatorname{Ind}^2(c)\neq 0$  , then $c$ is not a cosmetic crossing. 
\end{theorem}

\begin{proof} 
Let $c$ be a classical crossing of a twisted knot diagram $D$ such that $$\operatorname{Ind}^1(c)+\operatorname{Ind}^2(c)\neq 0.$$ Let $D'$ be the diagram obtained from $D$ by performing a crossing change at $c$, and let $c'$ denote the corresponding crossing in $D'$.

 \noindent Then, by Lemma~\ref{lem:cosmetic2}, we have 
\begin{equation}\label{eq:cos}
\begin{split}
\mathcal{B}_{D'}(t) - \operatorname{sgn}(c') \left(t^{\operatorname{Ind}^1(c')+\operatorname{Ind}^2(c')}  -  1 \right) \qquad \qquad \qquad \\  \qquad \qquad \qquad
=  \mathcal{B}_{D}(t) - \operatorname{sgn}(c) \left(t^{\operatorname{Ind}^1(c)+\operatorname{Ind}^2(c)}  -  1 \right) . 
\end{split}
 \end{equation}
 
\noindent It is evident that $\operatorname{sgn} (c') = - \operatorname{sgn} (c)$. Moreover, it follows from Lemma~\ref{lem:cosmetic} that              
\[\operatorname{Ind}^1 (c') =  \operatorname{Ind}^2 (c), \quad \text{and} \quad \operatorname{Ind}^2 (c') =  \operatorname{Ind}^1 (c).\]  
Hence, Eq.~\ref{eq:cos} implies that
\begin{align*}
\mathcal{B}_{D'}(t) -  \mathcal{B}_{D}(t)& =-\operatorname{sgn}(c) \left(t^{\operatorname{Ind}^1(c)+\operatorname{Ind}^2(c)} - 1 \right)-\operatorname{sgn}(c)\left(t^{\operatorname{Ind}^2(c)+\operatorname{Ind}^1(c)}  -  1 \right) \\
&= - \operatorname{sgn}(c) \left(t^{\operatorname{Ind}^1(c)+\operatorname{Ind}^2(c)} + t^{\operatorname{Ind}^2(c)+\operatorname{Ind}^1(c)}  -  2 \right).
\end{align*}
Clearly, if $\operatorname{Ind}^1 (c)+\operatorname{Ind}^2 (c) \neq 0$, then $\mathcal{B}_{D'}(t) \neq  \mathcal{B}_{D}(t)$. 
Hence, $D'$ is not equivalent to $D$. Therefore, the crossing $c$ is not cosmetic.  
\end{proof} 

\section*{Acknowledgments} 
First author was supported by the National Board of Higher Mathematics (NBHM)(Ref. No. 0204/16(11)/2022/R $\&$D-II/11983), Government of India. Second author acknowledges the support given by Anusandhan National Research
Foundation (ANRF) with sanction order no. CRG/2023/004921. 
The third author acknowledges the support of the University Grants Commission
(UGC), India, for a research fellowship with NTA Ref. No.231610035955 .

\end{document}